\documentclass[utf8]{frontiersFPHY} 

\usepackage{url,hyperref,lineno,microtype,subcaption}
\usepackage[onehalfspacing]{setspace}
\usepackage{bm}
\usepackage{amsthm}
\usepackage{graphicx}
\newtheorem{thm}{\protect\theoremname}
\theoremstyle{definition}

\theoremstyle{plain}
\newtheorem{prop}[thm]{Proposition}
\theoremstyle{plain}

\theoremstyle{plain}

\theoremstyle{remark}

\def\keyFont{\fontsize{8}{11}\helveticabold }
\def\firstAuthorLast{Berner {et~al.}} 
\def\Authors{Rico Berner\,$^{1,2,*}$ and Serhiy Yanchuk\,$^{1}$}
\def\Address{$^{1}$Institut f\"ur Mathematik, Technische Universit\"at Berlin, Strasse des 17. Juni 136, 10623 Berlin, Germany\\
$^{2}$Institut f\"ur Theoretische Physik, Technische Universit\"at Berlin, Hardenbergstr. 36, 10623 Berlin, Germany
}
\def\corrAuthor{Rico Berner}
\def\corrEmail{rico.berner@physik.tu-berlin.de}

\begin{document}
\onecolumn
\firstpage{1}

\title[Synchronization in networks with distance-dependent plasticity]{Synchronization in networks with heterogeneous adaptation rules and applications to distance-dependent synaptic plasticity} 

\author[\firstAuthorLast ]{\Authors} 
\address{} 
\correspondance{} 

\extraAuth{}

\maketitle

\newcommand{\rem}[2]{\MakeUppercase{\bf [#1: #2]} }

\begin{abstract}
This work introduces a methodology for studying synchronization in adaptive networks with heterogeneous plasticity (adaptation) rules. As a paradigmatic model, we consider a network of adaptively coupled phase oscillators with distance-dependent adaptations. For this system, we extend the master stability function approach to adaptive networks with heterogeneous adaptation. Our method allows for separating the contributions of network structure, local node dynamics, and  heterogeneous adaptation in determining synchronization. Utilizing our proposed methodology, we explain mechanisms leading to synchronization or desynchronization by enhanced long-range connections in nonlocally coupled ring networks and networks with Gaussian distance-dependent coupling weights equipped with a biologically motivated plasticity rule. 

\tiny
 \keyFont{ \section{Keywords:} synaptic plasticity, adaptive networks, phase oscillator, synchronization, distance-dependent synaptic plasticity, nonlocally coupled rings, master stability approach} 
\end{abstract}

\section{Introduction}\label{sec:intro}
In nature and technology, complex networks serve as a ubiquitous paradigm with a broad range of applications from physics, chemistry, biology, neuroscience, socioeconomic, and other systems~\cite{NEW03}. Dynamical networks consist of interacting dynamical units, such as neurons or lasers. Collective behavior in dynamical networks has attracted much attention in recent decades. Depending on the network and the specific dynamical system, various synchronization patterns with increasing complexity were explored~\cite{PIK01,STR01a,ARE08,BOC18}. Even in simple models of coupled oscillators, patterns such as complete synchronization~\cite{KUR84,PEC97}, cluster synchronization~\cite{YAN01a,CHO09,BEL11a,ZHA20}, and various forms of partial synchronization have been found, such as frequency clusters~\cite{BER19}, solitary~\cite{JAR18,TEI19,BER20c}, or chimera states~\cite{KUR02a,ABR04,SCH16b,OME18a,OME19c}. In particular, synchronization is believed to play a crucial role in brain networks, for example, under normal conditions in the context of cognition and learning~\cite{SIN99,FEL11}, and under pathological conditions, such as Parkinson's disease~\cite{HAM07,GOR20,PFE21}, epilepsy~\cite{JIR13,JIR14,AND16,GER20}, tinnitus~\cite{TAS12,TAS12a}, schizophrenia, to name a few~\cite{UHL09}.

The powerful methodology of master stability function~\cite{PEC98} has been a milestone for the analysis of synchronization phenomena. This method allows for the separation of dynamic and structural features in dynamical networks. It greatly simplifies the problem by reducing the dimension and unifying the synchronization study for different networks. Since its introduction, the master stability approach has been extended and refined for various complex systems~\cite{FLU10b,DAH12,KEA12,KYR14,LEH15b,TAN19,BER20,BOE20,MUL20}, and methods beyond the local stability analysis have been developed~\cite{BEL04,BEL05b,BEL06,BEL06a,DAL20}. More recently, the master stability approach has been extended to another class of oscillator networks with high application potential, namely adaptive networks~\cite{BER20b}.

Adaptive networks are commonly used models 
for various systems from nature and technology~\cite{JAI01,PRO05a,GRO06b,MAR17b,KUE19a,HOR20,MEI09a,MIK13,MIK14}. A prominent example are neuronal networks with spike-timing dependent plasticity, in which the synaptic coupling between neurons changes depending on their relative spiking times~\cite{MAR97a,ABB00,CAP08a,POP13}. There are a large number of studies investigating the dynamic properties induced by this form of synaptic plasticity~\cite{ZEN15}. However, analysis is usually limited to only one or two forms of spike timing-dependent plasticity within a neuronal population. On the other hand, experimental studies indicate that different forms of spike timing-dependent plasticity may be present within a neuronal population, where the form depends on the connection structure between the axons and dendrites~\cite{TAZ20}. Among all structural aspects, an important factor for the specific form of the plasticity rule is the distance between neurons~\cite{FRO05a,SJO06,FRO10}. More specifically, it has been found that the plasticity rule between proximal or distal neurons, respectively, can change from Hebbian-like to anti-Hebbian-like~\cite{LET06,MEI20a}.

This work introduces a methodology to study synchronization in adaptive networks with heterogeneous plasticity (adaptation) rules. As a paradigmatic system, we consider an adaptively coupled phase oscillator network~\cite{AOK09,KAS17,KAS18a,BER19a,BER21a,FEK20,FRA20}, which is proven to be useful for predicting and describing phenomena occurring in more realistic and detailed models~\cite{POP15,LUE16,CHA17a,ROE19a}. 
More specifically, in the spirit of the master stability function approach, we consider the synchronization problem as the interplay between network structure and a heterogeneous adaptation rule arising from distance- (or location-)dependent synaptic plasticity. For a given heterogeneous adaptation rule, our master stability function provides synchronization criteria for any coupling configuration. As illustrative examples, we consider a nonlocally coupled ring with biologically motivated plasticity rule, and a network with a Gaussian distance-dependent coupling weights. We explained such intriguing effects as synchronization or desynchronization by enhancement of long-distance links. 

We introduce the model in section~\ref{sec:model}. Building on findings from~\cite{BER20b}, we develop a master stability approach in section~\ref{sec:msf} that takes a heterogeneous adaptation rule in account. In section~\ref{sec:EVApprox}, we provide an approximation of the structural eigenvalues that determine the stability of the synchronous state. We then consider two different setups: a nonlocally coupled ring in section~\ref{sec:NonLocalRing} and a weighted network with Gaussian distance distribution of coupling weights in section~\ref{sec:Gaussian}. Both systems are equipped with a biologically motivated plasticity rule. In section~\ref{sec:conclusions}, we summarize the results.
\section{Model}\label{sec:model}
In this work, we study the synchronization on networks with adaptive coupling weights, where the adaptation (plasticity) rule depends on the distance between oscillators (neurons). We consider the model of adaptively coupled phase oscillators, which has proven to be useful for understanding dynamics in neuronal systems with spike timing-dependent plasticity~\cite{LUE16,ROE19a,BER20b}. The model reads as follows:
\begin{align}
	\frac{\mathrm{d}}{\mathrm{d}t}{{\phi}}_i &= \omega + \sum_{j=1}^N a_{ij} \kappa_{ij} g(\phi_i - \phi_j), \label{eq:PO_DDP_phi}\\
	\frac{\mathrm{d}}{\mathrm{d}t}{\kappa}_{ij} &= -\epsilon\left(\kappa_{ij}+h_{ij}(\phi_i-\phi_j)\right), \label{eq:PO_DDP_kappa}
\end{align}
where $\phi_i\in S^1=\mathbb{R}/2\pi\mathbb{Z}$ ($i=1,\dots,N$) is the phase of the $i$th oscillator, $\kappa_{ij}$ ($i,j=1,\dots,N$) is the dynamical coupling weight from oscillator $j$ to $i$, $\omega$ denotes the natural frequency of each oscillator, and $a_{ij}\in[0,1]$ are the entries of the weighted adjacency matrix $A$ describing the network connectivity. The time scales of the "fast" phase oscillators and "slow" coupling weights are separated by the parameter $\epsilon$, which we assume to be small $0<\epsilon\ll 1$. The functions $g$ and $h_{ij}$ denote the coupling and the $N^2$ plasticity functions, respectively. For illustrative purposes, the coupling function is set throughout the paper to $g(\phi) =-\sin(\phi+\alpha)/N$  with the phase lag parameter $\alpha$~\cite{SAK86}. Such a phase lag can account for a small synaptic propagation delay~\cite{ASL17,ASL18a}. For formal derivations, however, a generic coupling function is used. Note that the system~\eqref{eq:PO_DDP_phi}--\eqref{eq:PO_DDP_kappa} is shift-symmetric, i.e., invariant under the transformation $\phi_i \mapsto \phi_i + \psi$ for any $\psi\in S^1$. This allows us to restrict our consideration to the case $\omega=0$ by introducing a new  "co-rotating" coordinate system $\phi_{i,\text{new}} = \psi_i - \omega t$.

The main difference of system \eqref{eq:PO_DDP_phi}--\eqref{eq:PO_DDP_kappa} from the models considered previously in the literature~\cite{KAS17,KAS18a,BER20,BER21b,FEK20}, is that the plasticity functions $h_{ij}$ can be different for each network connection $j\to i$.

A solution to~\eqref{eq:PO_DDP_phi}--\eqref{eq:PO_DDP_kappa} is called  \textit{phase-locked} if, for all $i=1,\dots,N$, the phases evolve as $\phi_i = \Omega t + \vartheta_i$ with some \textit{collective frequency} $\Omega\in \mathbb{R}$ and $\vartheta_i\in S^1$. If $\vartheta_i=\vartheta$ for all $i=1,\dots,N$, the phase-locked state is called \textit{in-phase synchronous} or, short, \textit{synchronous} state.

In the case of in-phase synchronous state, we can set $\vartheta_i=0$ for each oscillator due to the shift symmetry of \eqref{eq:PO_DDP_phi}--\eqref{eq:PO_DDP_kappa}. The in-phase synchronous state is given as
\begin{align}
	{{\phi}}^s(t) &= -  w g(0)t, \label{eq:PO_DDP_sync_phi}\\
	{\kappa}^s_{ij} &= - h_{ij}(0), \label{eq:PO_DDP_sync_kappa}
\end{align}
where we assume that the weighted row sum $w=\sum_{j=1}^N a_{ij} h_{ij}(0)$ is constant for all $i=1\dots,N$. Such an assumption of constant row sum  is necessary for the existence of the synchronous state. Moreover, it is satisfied for commonly considered cases of global or nonlocal shift-invariant coupling.

In the following section, we show how the stability of the synchronous state is determined in a master-stability-like approach.  

\section{Master stability approach}\label{sec:msf}
In section~\ref{sec:model}, we have introduced a general class of models and the synchronous state, that are considered throughout this paper. In this section, we derive a framework for the local stability analysis of the synchronous states.
We note that the master stability approach for homogeneous adaptations $h_{ij}=h$ was introduced in \cite{BER20b,VOC21}. Here we extend the methodology to heterogeneous adaptation rules.

To describe the local stability, we introduce the variations ${\xi}_i = \phi_i-\phi^s$ and $\chi_{ij}=\kappa_{ij}-{\kappa}^{s}_{ij}$. The linearized equations for these variations can be written in the following matrix form 
\begin{align}\label{eq:VarEqs}
	\frac{\mathrm{d}}{\mathrm{d}t}
	\begin{pmatrix}
		\bm{\xi} \\
		\bm{\chi}
	\end{pmatrix}=
	{J}
	\begin{pmatrix}
		{\bm{\xi}} \\
		{\bm{\chi}}
	\end{pmatrix}
	=\begin{pmatrix}
		\mathrm{D}g(0) L^{h} & g(0) B \\ 
		-\epsilon C & -\epsilon\mathbb{I}_{N^2}
	\end{pmatrix}
	\begin{pmatrix}
		{\bm{\xi}} \\
		{\bm{\chi}}
	\end{pmatrix},
\end{align}
where $\bm{\xi}=(\xi_1,\dots,\xi_N)^T$ is $N$-dimensional vector containing the perturbations $\xi_i=\delta\phi_i$ of the phases and $\bm{\chi}=(\chi_{11},\chi_{12},\dots,\chi_{NN})^T$  are $N^2$-dimensional vectorized perturbations of coupling weights $\bm{\chi} = \bm{\mathrm{vec}}\left[\delta \kappa_{ij}\right]$, respectively.
The $N\times N$ weighted Laplacian matrix $L^{h}$ has the following elements
\begin{align}\label{eq:POLaplacianH}
	l^{h}_{ij}=\begin{cases}
		- \sum_{m=1,m\ne i}^N a_{im} h_{im}(0) , & i=j,\\
		a_{ij} h_{ij}(0) , & i\ne j.
	\end{cases}
\end{align}
The time-independent matrices $B$ and $C$ are
\begin{align*}
	B= \begin{pmatrix}
		\bm{a}_{1} & & \\
		& \ddots & \\
		& & \bm{a}_{N}
	\end{pmatrix}, \quad 
	C=\begin{pmatrix}
		(\mathrm{D}h)^T_{1} & & \\
		& \ddots & \\
		& & (\mathrm{D}h)^T_{N}
	\end{pmatrix}-
	\begin{pmatrix}
		\mathrm{diag}\,(\mathrm{D}h)_{1} \\
		\vdots \\
		\mathrm{diag}\,(\mathrm{D}h)_N
	\end{pmatrix},
\end{align*}
where $\bm{a}_i=(a_{i1},\dots,a_{iN})$,
$(\mathrm{D}h)_i=(\mathrm{D}h_{i1}(0),\dots,\mathrm{D}h_{iN}(0))$, and
\begin{align*}
	\mathrm{diag}\,( \mathrm{D}h(0))_{i} &= \begin{pmatrix}
		\mathrm{D}h_{i1}(0) & & \\
		& \ddots & \\
		& & \mathrm{D}h_{iN}(0)
	\end{pmatrix}.
\end{align*}
Note that due to the shift symmetry of~\eqref{eq:PO_DDP_phi}--\eqref{eq:PO_DDP_kappa}, the Jacobian $J$ in~\eqref{eq:VarEqs} is time independent. Therefore, the real parts of the $N(N+1)$ eigenvalues $\lambda$ of $J$ are the Lyapunov exponents of the synchronous state and hence determine its local stability. In the following proposition, we exploit the fact that $J$ contains a large diagonal block $-\epsilon \mathbb{I}_{N^2}$ to reduce the dimension of the eigenvalue problem for $J$.

\begin{prop}\label{prop:JacobianDimRed}
	Suppose $\phi_i = \Omega t$ is an in-phase synchronous state of~\eqref{eq:PO_DDP_phi}--\eqref{eq:PO_DDP_kappa}. Then its linear stability is determined by the $2N$-dimensional linear system
\begin{align}\label{eq:epsilonReduction}
	\frac{\mathrm{d}}{\mathrm{d}t}
\bm{v}
	=
	\begin{pmatrix}
		\mathrm{D}g(0)L^h & g(0) \mathbb{I}_N \\
		\epsilon L^{\mathrm{D}h} & -\epsilon \mathbb{I}_N
	\end{pmatrix}
\bm{v},
\end{align}
	where $\mathrm{D}g(0)$ and $L^h$ are as in~\eqref{eq:VarEqs} and the $N\times N$ weighted Laplacian matrix $L^{\mathrm{D}h}$ possesses the following elements
	\begin{align}\label{eq:POLaplacianDH}
		l^{\mathrm{D}h}_{ij}=\begin{cases}
			- \sum_{m=1,m\ne i}^N a_{im} \mathrm{D}h_{im}(0) , & i=j,\\
			a_{ij} \mathrm{D}h_{ij}(0), & i\ne j.
		\end{cases}
	\end{align}
\end{prop}
\begin{proof}
We remind that system~(\ref{eq:VarEqs}) determines the spectrum (Lyapunov exponents) of the synchronous state. The Jacobian matrix in~(\ref{eq:VarEqs}) is sparse with a large $N^2\times N^2$ block given by the simple diagonal matrix $-\epsilon\mathbb{I}_{N^2}$. This implies that~(\ref{eq:VarEqs}) possess $N^2-N$ stable directions with Lyapunov exponents $-\epsilon$. To find these directions, we substitute 
$(\bm{\xi}, \bm{\chi})=e^{-\epsilon t}(\bm{\xi}_0 , \bm{\chi}_0)$ into (\ref{eq:VarEqs}) and obtain the linear system 
\begin{align}
	\label{matrix000}
	\begin{pmatrix}
		\mathrm{D}g(0) L^h + \epsilon \mathbb{I}_N & g(0) B  \\ 
		-\epsilon C & 0
	\end{pmatrix}
	\begin{pmatrix}
		{\bm{\xi}_0} \\
		{\bm{\chi}_0}
	\end{pmatrix} = 0.
\end{align}
This system has at least $N^2-N$ linearly independent solutions, since the matrix in~\eqref{matrix000} is degenerate due to the large $N^2\times N^2$ zero block. The structure of the invariant subspaces in system~(\ref{eq:VarEqs}) allows for introducing new coordinates, which separate the $N^2-N$ stable directions (corresponding to the eigenvalues $-\epsilon$) from the remaining $2N$ directions. Explicitly, this transformation is given by 
$$
\begin{pmatrix}
{\bm{\xi}} \\
{\bm{\chi}}
\end{pmatrix}
=
R\begin{pmatrix}
{{\bm{\xi}}} \\
\hat{{\bm{\chi}}}
\end{pmatrix},\quad 
	R = \begin{pmatrix}
\mathbb{I}_{N} & 0 & 0\\
0 & (1/r)B^T & K
\end{pmatrix}
$$
with
$(N^2+N)\times (N^2+N)$ matrix $R$. Here $K$ is an $(N^2-N)\times (N^2-N)$ orthogonal matrix with $B K=0$. 
Applying this transformation, we obtain the
following system
\begin{align}\label{eq:epsilonReduction-1}
	\frac{\mathrm{d}}{\mathrm{d}t}
	\begin{pmatrix}
		\bm{\xi} \\
		\bm{\chi}_N \\
		\bm{\chi}_{N^2-N}
	\end{pmatrix}
	=
	\begin{pmatrix}
		\mathrm{D}g(0)L^h & g(0) \mathbb{I}_N & 0 \\
		\epsilon L^{\mathrm{D}h} & -\epsilon \mathbb{I}_N & 0 \\
	-\epsilon K^T C & 0 & -\epsilon \mathbb{I}_{N^2-N}		
	\end{pmatrix}
	\begin{pmatrix}
	\bm{\xi} \\
	\bm{\chi}_N \\
	\bm{\chi}_{N^2-N}
\end{pmatrix},
\end{align}
where
$(\bm{\xi}, \bm{\chi}_N , \bm{\chi}_{N^2-N})^T =
(\bm{\xi} , \bm{\hat \chi} )^T,$
with ${\bm{\chi}}_N$ and $\bm{\chi}_{N^2-N}$ are an $N$ and $N^2-N$-dimensional vectors, respectively, and the $N\times N$ weighted Laplacian matrix $L^{\mathrm{D}h}$ as given in~\eqref{eq:POLaplacianDH}. 
For more details on the transformation, we refer the reader to~\cite{BER20b,VOC21}. 
We observe that the variables $(\bm{\xi}, \bm{\chi}_N)$ are independent on $\bm{\chi}_{N^2-N}$. Hence, separating the master from the slave system, the resulting coupled differential equations that determine the stability of the synchronous state are given by system \eqref{eq:epsilonReduction}. This concludes the proof.
\end{proof}

Proposition \ref{prop:JacobianDimRed} reduces the problem's dimension significantly from $N(N+1)$ to $2N$. In the spirit of the master stability approach~\cite{PEC98}, we aim for further decomposition of the $2N$-dimensional coupled system~\eqref{eq:epsilonReduction} into dynamically independent blocks of dimension $2$. For this, we restrict our consideration to the case when $L^h$ can be diagonalized $S^h= Q^{-1} L^h Q$ by a nonsingular complex-valued matrix $Q$. Note that the eigenvalues $\mu_i$ of $L^h$ lie on the diagonal of $S^h$. In general, the matrices $L^h$ and $L^{\mathrm{D}h}$ do not commute. Therefore, $Q^{-1} L^{\mathrm{D}h}Q$ is not necessarily of upper triangular shape. Regardless of this fact, the following proposition provides an explicit form for the eigenvalues of $J$ in~\eqref{eq:VarEqs} in the limit of slow adaptation, i.e., $\epsilon \ll 1$. 

\begin{prop}\label{prop:JacobianEVApprox}
	Assume that $L^h$ is diagonalizable, with $S^h= Q^{-1} L^h Q$ being the associated diagonal matrix and $Q$ the corresponding transformation. Let $\phi_i = \Omega t$ be an in-phase synchronous state of~\eqref{eq:PO_DDP_phi}--\eqref{eq:PO_DDP_kappa}. Then, the local stability of this state is determined by the solutions of $N$ quadratic equations, which are given up to the first order in $\epsilon$ as
	\begin{align}\label{eq:masterStabEquation}
		\lambda^2 + (\epsilon-\mathrm{D}g(0)\mu_i)\lambda-\epsilon\left(\mathrm{D}g(0)\mu_i + g(0) \nu_i \right) = 0,  \quad i=1,\dots,N,
	\end{align}
	where $\mu_i$ are the eigenvalues of $L^h$ located on the diagonal of $S^h$ and $\nu_i$ are the corresponding diagonal elements of $Q^{-1} L^{\mathrm{D}h}Q$. 
	
	If $L^h$ and $L^{\mathrm{D}h}$ commute,  then Eq.~\eqref{eq:masterStabEquation} is exact, and $\nu_i$ are the eigenvalues of $L^{\mathrm{D}h}$.
\end{prop}
\begin{proof}
Due to Proposition~\ref{prop:JacobianDimRed}, the eigenvalues of the Jacobian in~\eqref{eq:VarEqs} are given by
\begin{align*}
	\det\begin{pmatrix}
		\mathrm{D}g(0)L^h-\lambda\mathbb{I}_N & g(0) \mathbb{I}_N \\
		\epsilon L^{\mathrm{D}h} & -(\epsilon+\lambda) \mathbb{I}_N
	\end{pmatrix}=
	\det\begin{pmatrix}
		\mathrm{D}g(0) S^h-\lambda\mathbb{I}_N & g(0) \mathbb{I}_N \\
		\epsilon Q^{-1} L^{\mathrm{D}h}Q & -(\epsilon+\lambda) \mathbb{I}_N
	\end{pmatrix} = 0,
\end{align*}
where we have used the transformation $Q$ that brings $L^h$ to the diagonal form $S^h= Q^{-1} L^h Q$.
Making further use of the Schur complement~\cite{LIE15},
we obtain 
\begin{align}\label{eq:JacobianEVApproxAux1}
	\det\begin{pmatrix}
		\mathrm{D}g(0) S^h-\lambda\mathbb{I}_N & g(0) \mathbb{I}_N \\
		\epsilon Q^{-1} L^{\mathrm{D}h}Q & -(\epsilon+\lambda) \mathbb{I}_N
	\end{pmatrix} = \det\begin{pmatrix}
		(\lambda+\epsilon)\left(\lambda\mathbb{I}_N-\mathrm{D}g(0) S^h\right) - \epsilon g(0) Q^{-1} L^{\mathrm{D}h}Q \end{pmatrix} = 0.
\end{align}
The latter equation is almost diagonal. The only off-diagonal components remain from $Q^{-1} L^{\mathrm{D}h}Q$ and scale with $\epsilon$. Let us consider the Leibniz formula for the determinant of an $N\times N$ matrix $F$ with entries $f_{ij}$, that reads $\det(F)=\sum_{\sigma\in \mathrm{Perm}(N)}\mathrm{sgn}(\sigma)\prod_{i=1}^{N}f_{i\sigma(i)}$. In the latter expression $\mathrm{Perm}(N)$ denotes the set of all permutations $\sigma$ of the integer numbers $1,\dots,N$ and $\mathrm{sgn}(\sigma)\in\{-1,1\}$ is the sign of the permutation. Since all off-diagonal terms of the matrix considered in~\eqref{eq:JacobianEVApproxAux1} scale with $\epsilon$, for any but the identical permutation each term $\prod_{i=1}^{N}f_{i\sigma(i)}$ scales with $\epsilon^2$ or higher. Hence, we are left with $\det(F)=\prod_{i=1}^{N}f_{ii} + \mathcal{O}(\epsilon^2)$ and find 
\begin{multline}\label{eq:JacobianEVApproxAux2}
	\det\begin{pmatrix}
		(\lambda+\epsilon)\left(\lambda\mathbb{I}_N-\mathrm{D}g(0) S^h\right) - \epsilon g(0) Q^{-1} L^{\mathrm{D}h}Q \end{pmatrix} \\ 
	= \prod_{i=1}\left(\lambda^2 + (\epsilon-\mathrm{D}g(0)\mu_i)\lambda-\epsilon\left(\mathrm{D}g(0)\mu_i + g(0) \nu_i \right)\right) + \mathcal{O}(\epsilon^2) = 0,
\end{multline}
where $\mu_i$ are the eigenvalues of $L^h$, $\nu_i$ are the diagonal elements of $Q^{-1} L^{\mathrm{D}h}Q$ and $\mathcal{O}(\epsilon^2)$ denotes higher order terms ($\epsilon^m, m>1$). If $L^h$ and $L^{\mathrm{D}h}$ commute, both matrices share the same set of eigenvectors and hence they can be brought to the diagonal form with the same transformation $Q$. In this case, the diagonal elements $\nu_i$ are the eigenvalues of $L^{\mathrm{D}h}$ and the higher order terms $\mathcal{O}(\epsilon^2)$ in~\eqref{eq:JacobianEVApproxAux2} vanish.
\end{proof}

The $2N$ solutions $\lambda_i$ of the $N$ equations~\eqref{eq:masterStabEquation} determine the stability of the synchronous state. More precisely, the real parts of theses solutions determine the Lyapunov exponents. 
If $\Lambda=\max_i \mathrm{Re}(\lambda_i)<0$, then the synchronous state is locally stable, while for $\Lambda>0$ it is locally unstable. The case $\Lambda=0$ provides the stability boundary. 

Note that for a fixed time scale parameter $\epsilon\ll 1$, the equation~\eqref{eq:masterStabEquation} and hence its solutions depend on the coupling function $g$, the connectivity, and the adaptation structure. This dependence, however, is only encoded in the two complex parameters $\mathrm{D}g(0)\mu$ and $g(0)\nu$. Therefore, we define the master stability function $\Lambda:\mathbb{C}^2\to\mathbb{R}$ with $\Lambda(\mathrm{D}g(0)\mu,g(0)\nu)=\max_i \mathrm{Re}(\lambda_i(\mathrm{D}g(0)\mu,g(0)\nu))$ that maps each pair of parameters $(\mathrm{D}g(0)\mu,g(0)\nu)$ to the corresponding Lyapunov exponent.

\begin{figure}
	\centering
	\includegraphics{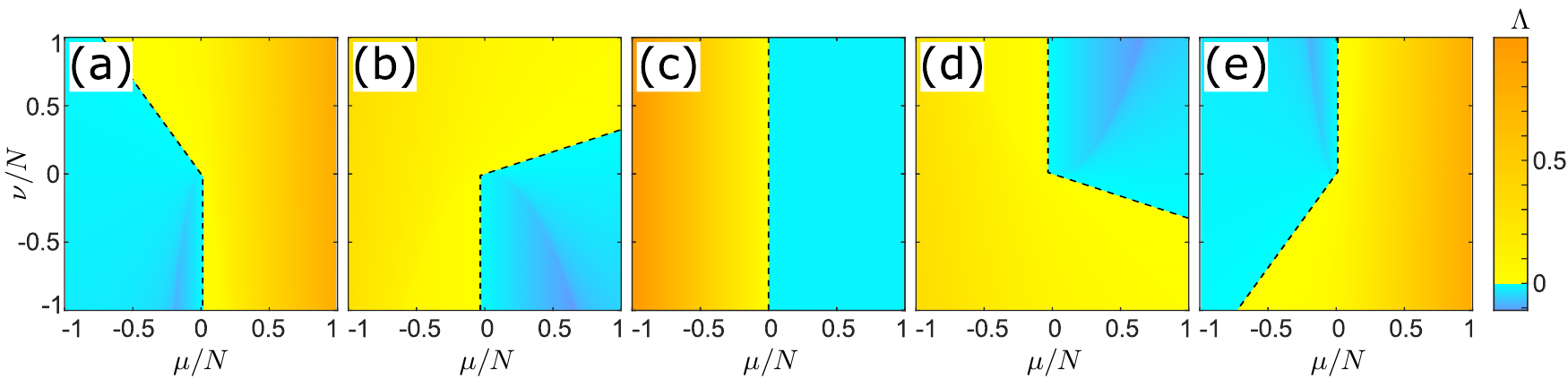}
	\caption{The master stability function $\Lambda(\mathrm{D}g(0)\mu,g(0)\nu)$ for the
        coupling function $g(\phi) =-\sin(\phi+\alpha)/N$ and real $\mu$ and $\nu$ ($\mathrm{Im}(\mu)=0$, $\mathrm{Im}(\nu)=0$). The values of the master stability function are color-coded in all panels (a--e). The dashed black line describes the border between regions corresponding to local stability and instability, respectively. Parameters: $\epsilon=0.01$, (a) $\alpha=-0.8\pi$, (b) $\alpha=-0.4\pi$, (c) $\alpha=0$, (d) $\alpha=0.4\pi$, and (e) $\alpha=0.8\pi$.}\label{fig:MSF}
\end{figure}
For an illustration, we consider a cross-section of $(\mathrm{D}g(0)\mu,g(0)\nu)$-space by setting $\mathrm{Im}(\mu)=0$ and $\mathrm{Im}(\nu)=0$. This cross-section is of particular interest in cases of symmetric matrices $L^h$ and $L^{\mathrm{D}h}$ since their eigenvalues are real. In figure~\ref{fig:MSF}, we present the master stability function for the coupling function $g(\phi)=-\sin(\phi+\alpha)/N$ and different values of the parameter $\alpha$. In case of real $\mu$ and $\nu$, we obtain two explicit stability conditions from~\eqref{eq:masterStabEquation}: The synchronous state is locally stable ($\Lambda<0$) if 
\begin{align}
	c_1(\alpha,\mu)&=\cos(\alpha)\mu >-\epsilon, \label{eq:synyStabCrit1}\\
	c_2(\alpha,\mu,\nu)&=\cos(\alpha)\mu + \sin(\alpha) \nu > 0. \label{eq:synyStabCrit2}
\end{align}
These conditions agree with the black dashed lines in figure~\ref{fig:MSF} and are used subsequently to describe stability for certain network models.

\section{Synchronization on networks with distance dependent plasticity}\label{sec:systDDP}

In the previous section, we established a generic analytic tool for studying stability of synchronous states. In this section, we focus on the application of the tool to certain network models. For the rest of the work, we restrict our attention to the following generalization of the Kuramoto-Sakaguchi system with distance-dependent synaptic plasticity
\begin{align}
	\frac{\mathrm{d}}{\mathrm{d}t}{{\phi}}_i &= \omega -\frac{1}{N} \sum_{j=1}^N a_{ij} \kappa_{ij} \sin(\phi_i - \phi_j+\alpha) \label{eq:AKS_DDP_phi}, \\
	\frac{\mathrm{d}}{\mathrm{d}t}{\kappa}_{ij} &= -\epsilon\left(\kappa_{ij}+h(\phi_i-\phi_j,d_{ij})\right). \label{eq:AKS_DDP_kappa}
\end{align}
The plasticity function $h$ depends on the phase difference $\phi_i-\phi_j$ and on the distance $d_{ij}$. In this work, we associate the distance to the difference of indices by $d_{ij}=|j-i|$. For the plasticity function, we consider
\begin{align}\label{eq:adaptationfunctionSymmetricOnRing}
	h_{ij}(\phi) = h(\phi,\frac{d_{ij}}{N}) = \begin{cases}
		\hat{h}(\phi,\frac{d_{ij}}{N}) & d_{ij} \le N/2,\\
		\hat{h}(\phi,1-\frac{d_{ij}}{N}) & d_{ij} > N/2.
	\end{cases}
\end{align}
With this form of the adaptation function, we have a symmetric $h_{ij}(\phi)=h_{ji}(\phi)$ and a circulant $h_{i+l,j+l}(\phi)=h_{ij}(\phi)$ structure of the corresponding matrix with entries $h_{ij}$. Particularly, for the numerical analysis, we use
\begin{equation}
\label{eq:plasticity}
\hat{h}(\phi,d_{ij}/N)=\sin(\phi+\beta(d_{ij}/N)),
\end{equation}
 where the distance dependence is encoded in the phase shift function
\begin{align}\label{eq:STDP_DDP_betaShift}
	\beta\left(\frac{d_{ij}}{N}\right) = \begin{cases}
		\left(\frac{2}{N}d_{ij}-1\right)\pi, & N \text{ even},\\
		\left(\frac{2}{(N+1)}d_{ij}-1\right)\pi, & N \text{ odd}.
	\end{cases}
\end{align}
\begin{figure}
	\centering
	\includegraphics{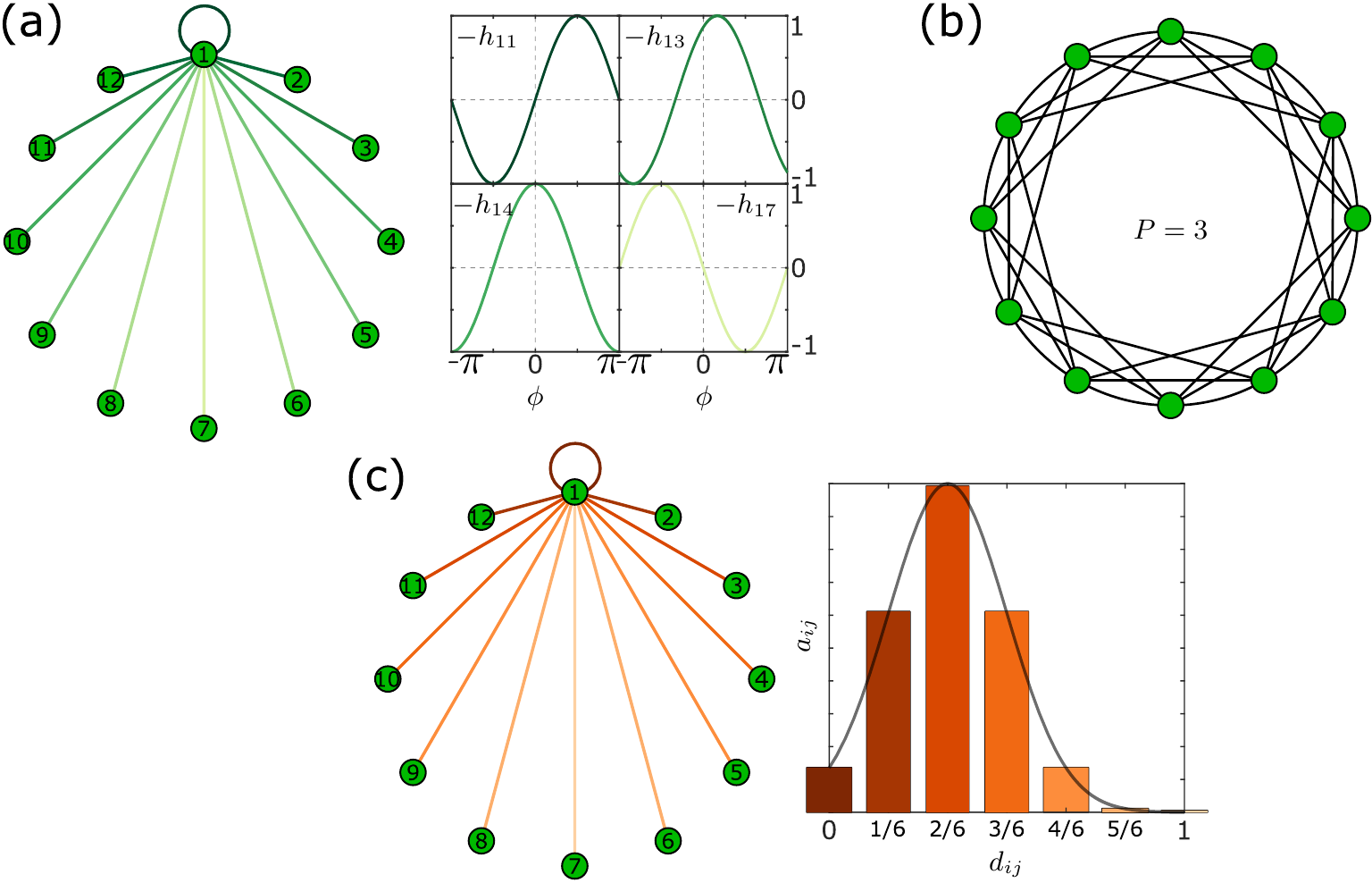}
	\caption{Panel (a) shows the plasticity function $h_{1j}$ given in ~\eqref{eq:adaptationfunctionSymmetricOnRing}--\eqref{eq:STDP_DDP_betaShift} depending on the distance $d_{1j}$ exemplified for node $i=1$ in a network with $N=12$ nodes. Note that the colors of the links in the network (left) correspond to the colors of the depicted plasticity function (right). Panel (b) displays the connectivity structure of a nonlocally coupled ring network with $N=12$ nodes and a coupling range $P=3$. Panel (c) displays the weighted connectivity structure of a network with $N=12$ nodes (left) with distance dependent Gaussian weight distribution (right).  Note that the colors of the links in the network (left) correspond to the colors of the bars in the weight distribution (right).\label{fig:illustrationDDP}}
\end{figure}

In figure~\ref{fig:illustrationDDP}(a), we illustrate the distance-dependent plasticity function ~\eqref{eq:adaptationfunctionSymmetricOnRing}--\eqref{eq:STDP_DDP_betaShift} for a network of $N=12$ nodes. The illustration shows the different plasticity functions depending on the distance between the nodes $d_{ij}$. The plasticity function changes from a Hebbian to anti-Hebbian rule for proximal and distal node, respectively. This change, particularly in the proximity of $\phi=0$, is in qualitative agreement with the experimental findings in~\cite{LET06}. Note the symmetry of the plasticity function 
that renders the matrix with elements $h_{ij}$ circulant.

If not indicated differently, we consider the coupling structure given by
\begin{equation}
	a_{ij} = a(d_{ij}/N),\label{eq:ringLike}
\end{equation}
where $a: [0,1]\to [0,1]$ is a bounded and piece-wise continuous function. This corresponds to a distant-dependent coupling, and it results to a dihedral symmetry in the coupling structure (ring-like).

In the following section, we provide an approximation for the eigenvalues of $L^h$ and $L^{\mathrm{Dh}}$ for large networks with circulant connectivity and plasticity structure. Using this approximation, we subsequently analyze the stability of the synchronous state on nonlocally coupled networks and on isotropic networks with Gaussian weight distribution.

\subsection{Approximation of the eigenvalues for large systems with circulant structure}\label{sec:EVApprox}

In the previous part, we have defined the plasticity functions $h_{ij}$ in such a way that the structures of $L^h$ and $L^{\mathrm{D}h}$ inherit important properties from the underlying network structure $a(d_{ij}/N)$. In particular, assuming that the adjacency matrix is circulant, renders $L^h$ and $L^{\mathrm{D}h}$ to be circulant, as well. 

In this section, we briefly recall how one can derive the eigenvalues $\mu_k$ and $\nu_k$ ($k=0,\dots,N-1$) in case of a circulant structure. It is well-known that for a circulant matrix the eigenvalues are determined by applying a discrete Fourier approach~\cite{GRA06}. More precisely, suppose $L$ is a circulant $N\times N$ matrix where the elements of the first row are given by the entries $l_j$ with $j=1,\dots,N$. Then the $k$th eigenvalue is explicitly given by
\begin{align*}
	\mu_k = l_1+\sum_{j=2}^{N}l_j \exp\left(\mathrm{i}\frac{2\pi}{N}(j-1)k\right).
\end{align*}
For the case of $L^h$ as in~\eqref{eq:POLaplacianH}, $a_{ij}$ and $h_{ij}$ as in~\eqref{eq:ringLike} and \eqref{eq:adaptationfunctionSymmetricOnRing}, we obtain
\begin{align}\label{eq:EVMu_circulantExplicit}
	\mathrm{Re}(\mu_k) = \mathrm{Re}(l_{11}^h) + \frac{1}{N}\sum_{j=2}^{N} a(x_j) h(0,x_{j})  \cos\left(2\pi x_{j}k\right),
\end{align}
with $x_j=d_{1j}/N$ and  $\mathrm{Re}(l_{11}^h)=-\frac{1}{N}\sum_{j=2}^{N} a(x_j) h(0,x_{j})$. Since the adjacency matrix $A$ is assumed to be symmetric, the eigenvalues of $L^h$ are real. Therefore, we omit considering the imaginary part of $\mu_k$. Equation~\eqref{eq:EVMu_circulantExplicit} provides exact expressions for the eigenvalues. However, the values depend on the total number of oscillators $N$ that makes it harder to study the influence of other system properties, such as the coupling structure or the plasticity function. To remove this $N$-dependence, we consider the continuum limit $N\to\infty$ (compare with~\cite{AOK11}) and obtain
\begin{align*}
	\mathrm{Re}(\mu_k) = \mathrm{Re}(l_{11}^h) + \int_{0}^{1}a(x) h(0,x) \cos\left(2\pi x k\right)\mathrm{d}x,
\end{align*}
Due to the definition of $h$ and the symmetry of $a(x)$, we find
\begin{align}\label{eq:EVApproxMu_general}
	\mathrm{Re}(\mu_k) = 2\int_{0}^{1/2} a(x) h(0,x) (\cos\left(2\pi x k\right)-1)\mathrm{d}x
\end{align}
for any $k$. This explicit expression allows studying the distribution of the eigenvalues $\mu_k$ for a given plasticity function $h$ and coupling structure $a$. Note that a similar expression as~\eqref{eq:EVApproxMu_general} can be analogously derived for the eigenvalues of $L^{\mathrm{D}h}$ and reads
\begin{align}\label{eq:EVApproxNu_general}
	\mathrm{Re}(\nu_k) = 2\int_{0}^{1/2} a(x)\mathrm{D}h(0,x) (\cos\left(2\pi x k\right)-1)\mathrm{d}x.
\end{align}
We note that $\mu_0=\nu_0=0$ due to the Laplacian structure of $L^h$ and $L^{\mathrm{D}h}$. 

The results from Eqs.~\eqref{eq:EVApproxMu_general} and~\eqref{eq:EVApproxNu_general} are applied in the next sections to analyze different networks.

\subsection{Synchronization on nonlocally coupled ring networks}\label{sec:NonLocalRing}

In this section, we analyze the effect of long distance connections on the stability of synchronous states in nonlocally coupled ring networks. We consider the coupling structure given by
\begin{equation}
	a_{ij} = a(d_{ij}/N)=\begin{cases}
		1\quad\mbox{for}\ \ 0<d_{ij}\le P,\\
		1\quad\mbox{for}\ \ 0<N-d_{ij}\le P,\\
		0\quad\mbox{otherwise.}
	\end{cases}\label{eq:ring}
\end{equation}
This means that any two oscillators are coupled if they are separated at most by the coupling range $P$. The coupling Eq.~(\ref{eq:ring}) defines a nonlocal ring structure with coupling range $P$ to each side and two special limiting cases: local ring for $P=1$ and globally coupled network for $P=N/2$ (if $N$ is even, else $P=(N+1)/2$). The matrix of the form (\ref{eq:ring}) is circulant~\cite{GRA06} and has constant row sum, i.e., $\sum_{j=1}^{N}a_{ij}=2P$ for all $i=1,\dots,N$. An illustration for $N=12$ adn $P=3$ is presented in figure~\ref{fig:illustrationDDP}(b). 

In order to study the influence of the coupling range, we use the approximations for the eigenvalues $\mu_k$ and $\nu_k$ derived in section~\ref{sec:EVApprox}. The nonlocally coupled ring structure is expressed by the piecewise continuous function $a(x)=0$ for $p<x<1-p$ and $a(x)=1$ otherwise with relative coupling range $p=P/N$. Thus, for a nonlocally coupled ring~\eqref{eq:ring} and plasticity function~\eqref{eq:adaptationfunctionSymmetricOnRing}--\eqref{eq:STDP_DDP_betaShift}, we find
\begin{multline}\label{eq:EVApproxMu_NonLocalRing}
	\mathrm{Re}(\mu_k)=-2\int_{0}^{p} \sin(2\pi x) (\cos\left(2\pi k x \right)-1)\mathrm{d}x\\
	=\frac{(1-\cos(2\pi p))}{\pi}+\frac{1}{\pi}\begin{cases}
		\frac12 (\cos^2(2\pi p)-1) & k=1 \\
		\frac{1}{(1-k^2)}(k \sin(2\pi p) \sin(2\pi k p)+\cos(2\pi p) \cos(2\pi k p)-1) & k\ne 1
	\end{cases}
\end{multline}
for the eigenvalues $\mu_k$ of $L^h$. Analogously, we obtain
\begin{multline}\label{eq:EVApproxNu_NonLocalRing}
	\mathrm{Re}(\nu_k)=-2\int_{0}^{p} \cos(2\pi x) (\cos\left(2\pi k x \right)-1)\mathrm{d}x\\
	=\frac{\sin(2\pi p)}{\pi}-\frac{1}{\pi}\begin{cases}
		p \pi+\frac{\sin(4\pi p)}{4} & k=1 \\
		\frac{1}{(1-k^2)}\left(\sin(2\pi p) \cos(2\pi k p)-k\cos(2\pi p) \sin(2\pi k p)\right) & k\ne 1
	\end{cases}
\end{multline}
for $\nu_k$ of $L^{\mathrm{D}h}$.

\begin{figure}
	\centering
	\includegraphics{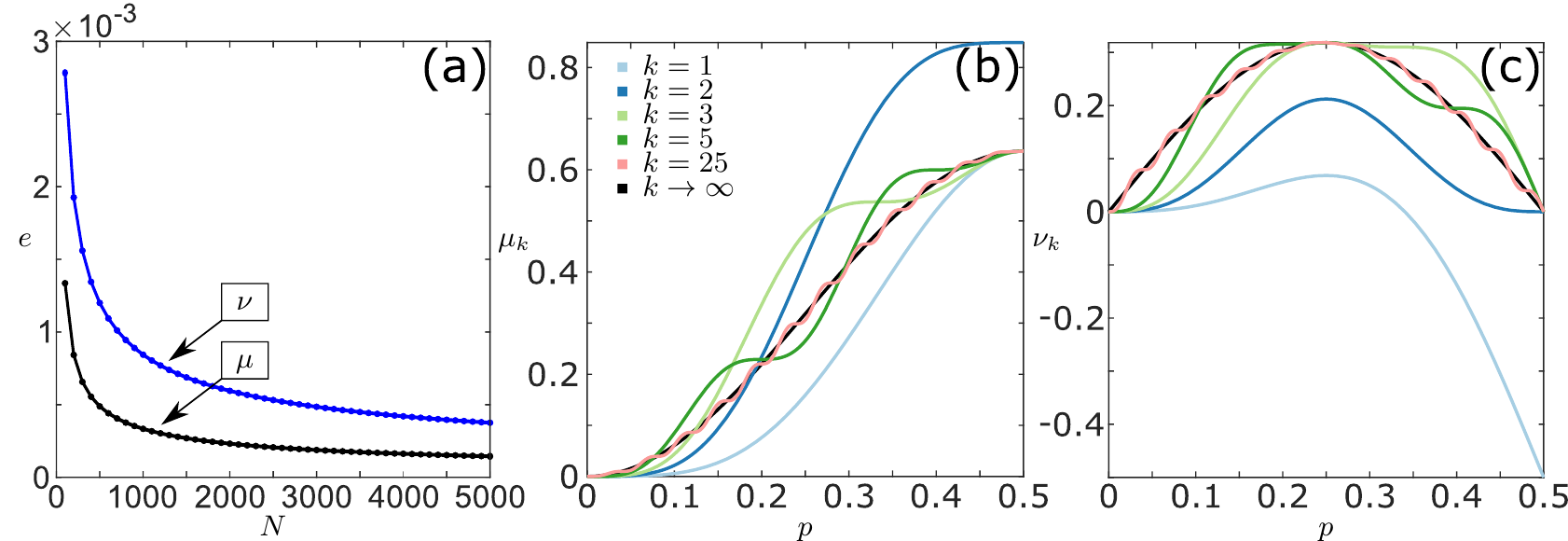}
	\caption{Panel (a) shows the errors $e(\mu)$ (black) and $e(\nu)$ (blue) with  $e(\gamma)=\sqrt{\frac{\sum_{k=0}^{N-1}(\gamma^{\text{exact}}_k-\gamma_k)^2}{N}}$ of the approximations ~\eqref{eq:EVApproxMu_NonLocalRing} and~\eqref{eq:EVApproxNu_NonLocalRing}, respectively, where $\gamma^{\text{exact}}_k$ are the exact eigenvalues derived by a discrete Fourier transformation, see~\eqref{eq:EVMu_circulantExplicit}. The errors are displayed in dependence of the system size $N$ (number of oscillators). The relative coupling range is set to $p=0.1$. Panel (b) and (c) show the approximated eigenvalues given by~\eqref{eq:EVApproxMu_NonLocalRing} and~\eqref{eq:EVApproxNu_NonLocalRing}, respectively, depending on the relative coupling range $p$ for different values of $k$.}\label{fig:EVApprox_Ring}
\end{figure}
In figure~\ref{fig:EVApprox_Ring}(a), we provide an error analysis of the approximations~\eqref{eq:EVApproxMu_NonLocalRing} and~\eqref{eq:EVApproxNu_NonLocalRing} compared to the exact eigenvalues given by Eq.~\eqref{eq:EVMu_circulantExplicit}. As expected, the errors tend to zero as the number of oscillators increases. Additionally in figures~\ref{fig:EVApprox_Ring}(b,c), we display $\mu_k$ and $\nu_k$ for several values of $k$ depending on the relative coupling range $p$. We observe that $\mu_k\ge 0$ for all $k$. This is due to given plasticity function~\eqref{eq:adaptationfunctionSymmetricOnRing}--\eqref{eq:STDP_DDP_betaShift}, for which the update is positive (or equal to zero) for all distances at $\phi=0$, i.e., $h(0,d_{ij})\ge 0$ for all $d_{ij}$. 

It is important to note, that our choice of the circulant adaptation functions imply that the matrices $L^h$ and $L^{\mathrm{D}h}$ are diagonalizable and commute. Hence, Proposition~\ref{prop:JacobianEVApprox} holds with the master stability equation~\eqref{eq:masterStabEquation} being exact. Therefore, the stability criterium \eqref{eq:synyStabCrit1} is also exact.

Combining the fact $\mu_k\ge 0$ with the stability criterium~\eqref{eq:synyStabCrit1}, we find $\cos(\alpha)>0$ as a necessary condition for the stability of the synchronous state for  $\epsilon\to 0$. This yields, that the synchronous state can be stable only for $\alpha\in(-\pi/2,\pi/2)$. In contrast to $L^h$, the $L^{\mathrm{D}h}$ is in general neither positive nor negative definite, hence the eigenvalues $\nu_k$ may take positive or negative values. This is due to the fact that the plasticity function may change sign at the origin, i.e., $\mathrm{D}h_{ij}$ may change signs depending on the distance $d_{ij}$. In particular, we find that only the eigenvalue $\nu_1$ changes the sign, see figure~\ref{fig:EVApprox_Ring}(c). This change may lead to a destabilization of the synchronous states as we show in the subsequent analysis. Finally, note that there exist $\mu_\infty=(1-\cos(2\pi p))/\pi$ and $\nu_\infty =-\sin(2\pi p)/\pi$ to which the eigenvalues converge for large values of $k$. These limits are displayed in figures~\ref{fig:EVApprox_Ring}(b,c) as black lines.

\begin{figure}
	\centering
	\includegraphics{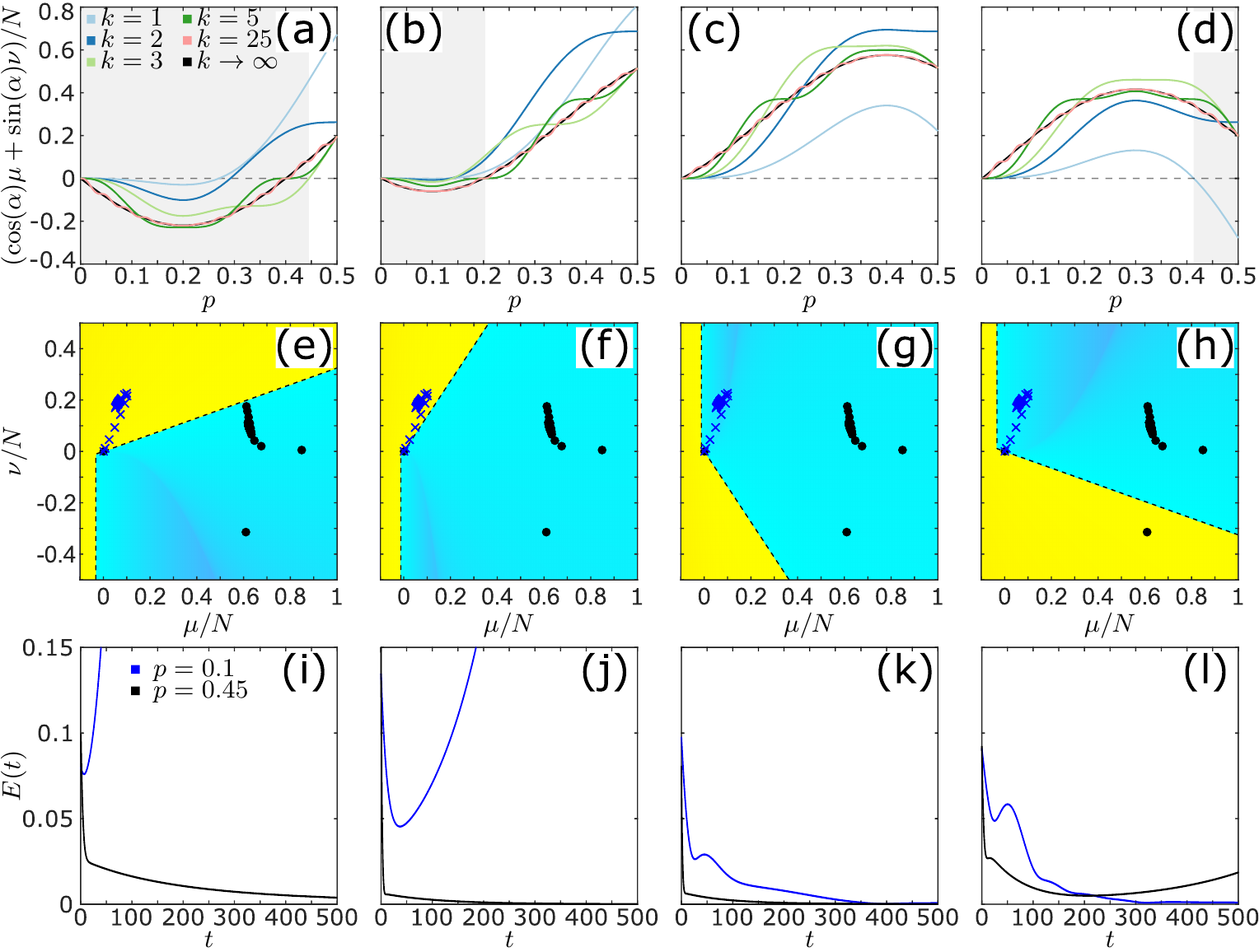}
	\caption{Stability analysis of the synchronous state of system~\eqref{eq:AKS_DDP_phi}--\eqref{eq:AKS_DDP_kappa} with plasticity rule~\eqref{eq:adaptationfunctionSymmetricOnRing}--\eqref{eq:STDP_DDP_betaShift} and coupling structure \eqref{eq:ring}. Panels (a,b,c,d) show  the function $c_2(\alpha,\mu_k(p),\nu_k(p))$ for different $\alpha$, see~\eqref{eq:synyStabCrit2}, calculated with the approximations~\eqref{eq:EVApproxMu_NonLocalRing} and~\eqref{eq:EVApproxNu_NonLocalRing} depending on the relative coupling range $p$. In each panel, $c_2$ is displayed for different values of $k$. The gray shaded regions refer to unstable synchronous states. Panels (e,f,g,h) show the master stability function $\Lambda(\mathrm{D}g(0)\mu,g(0)\nu)$ for the cross-section $\mathrm{Im}(\mu)=0$ and $\mathrm{Im}(\nu)=0$ for different values of $\alpha$ with color code as in figure~\ref{fig:MSF}. The crosses and dots correspond to two sets of eigenvalue pairs $(\mu_k,\nu_k)$ ($k=0,\dots,N-1$) for relative coupling range $p=0.1$ (blue crosses) and $p=0.45$ (black points), respectively. Panels (i,j,k,l) show the synchronization error $E(t)=\sqrt{\sum_{i=1}^N(\phi_i(t)-\phi_1(t))^2}$ for simulations with relative coupling range $p=0.1$ (blue) and $p=0.45$ (black). Each simulation is initialized at a slightly perturbed synchronous state. Parameters: $N=200$, $\epsilon=0.01$, (a,e,i) $\alpha=-0.4\pi$, (b,f,j) $\alpha=-0.2\pi$, (c,g,k) $\alpha=0.2\pi$, (d,h,l) $\alpha=0.4\pi$.}\label{fig:SyncStability_NonLocalRing}
\end{figure}
In figure~\ref{fig:SyncStability_NonLocalRing}, we show different scenarios for the stability of the synchronous state depending on the phase lag parameter $\alpha$ and the coupling range $p$. Due to the necessary condition $\cos(\alpha)>0$ as $\epsilon\to 0$, we consider $\alpha\in(-\pi/2,\pi,2)$ only. Figures~\ref{fig:SyncStability_NonLocalRing}(a) and (b) show that for $-\pi/2<\alpha<0$, the second stability condition~\eqref{eq:synyStabCrit2} is only fulfilled for $p$ larger than a critical value of the coupling range $p_c(\alpha)$. In these cases, a higher coupling range stabilizes the synchronous state. Note that $p_c(\alpha)\to 0$ as $\alpha\to 0$ with $\alpha<0$. The results seen in figure~\ref{fig:SyncStability_NonLocalRing}(a,b) are in agreement with the results for a network of $N=200$ coupled phase oscillators. For this network, we calculate the Laplacian eigenvalues and plot them along with the master stability function in figure~\ref{fig:SyncStability_NonLocalRing}(e,f). The outcomes from numerical simulations are presented in figure~\ref{fig:SyncStability_NonLocalRing}(i,j). 

The situation changes for $0<\alpha<\pi/2$, as shown in figure~\ref{fig:SyncStability_NonLocalRing}(c,d). Here, for a large range of $\alpha$, all nonlocally coupled networks lead to a stable synchronous state. However, closer to $\pi/2$, long distance connections destabilize the synchronous state. In particular, this destabilization can be traced back to the single negative eigenvalue $\nu_1$ of the Laplacian $L^{\mathrm{D}h}$, see figure~\ref{fig:SyncStability_NonLocalRing}(h). Hence, the unstable manifold of the synchronous state is only one-dimensional. This finding is in agreement with the example of $N=200$ phase oscillators presented in figure~\ref{fig:SyncStability_NonLocalRing}(g,h,k,l). Particularly in figure~\ref{fig:SyncStability_NonLocalRing}(l), the low dimension of the unstable manifold manifests itself as follows: The black trajectory first tends to the synchronous state along the $N(N+1)-1$ stable directions before it is repelled along the direction corresponding to $\nu_1$.

We have shown that long distance interactions may stabilize or destabilize the synchronous state depending on the phase lag parameter $\alpha$. In this section, all links have the same weight independent of the corresponding distance. In the next section, we analyze a network with a more realistic structure with a distance-dependent distribution of weights.

\vspace{0.5cm}

\subsection{Synchronization on isotropic and homogeneous network with Gaussian distance distribution}\label{sec:Gaussian}

In the previous section, we used the prototypical example of a nonlocally coupled rings to study the effects of long-range interaction on synchronization. In this setup, however, all links are equally weighted. In realistic systems, in contrast, the number of links with a certain distance are distributed, see~\cite{LET06} for details. To incorporate this into our network model, we weight the links with respect to a distance distribution. Measurements suggest that the distance distribution can be estimated by a mean and a distribution width~\cite{LET06}. The  Gaussian distributions is a paradigmatic distribution that allows for studying effects emanating from the mean and the distribution width. For the remainder of the section, we consider the link distance distribution given by a Gaussian distribution, and weight the links of the network connectivity structure $A$ accordingly, i.e.
\begin{align}\label{eq:GaussianWeightNetwork}
	a_{ij}(d_{ij}/N) =  \begin{cases}
		e^{-\frac{\left(d_{ij}/N-\xi\right)^2}{2\sigma^2}} & d_{ij} \le N/2,\\
		e^{-\frac{\left(1-d_{ij}/N-\xi\right)^2}{2\sigma^2}} & d_{ij} > N/2.
	\end{cases}
\end{align}
where $\xi$ and $\sigma$ are the mean value and the standard deviation, respectively. Note that the standard deviation characterizes the width of the distribution. For the numerical simulations, we normalize each row of $A$ by $\sum_{j=1}^N a_{ij}$. Here, we further make the assumption that the network is homogeneous and isotropic. This means that in any direction from a node and at each node the network looks the same. Hence, we obtain a circulant connectivity structure. An illustration of the weight distribution for $N=12$ is presented in figure~\ref{fig:illustrationDDP}(c).

\begin{figure}
	\centering
	\includegraphics{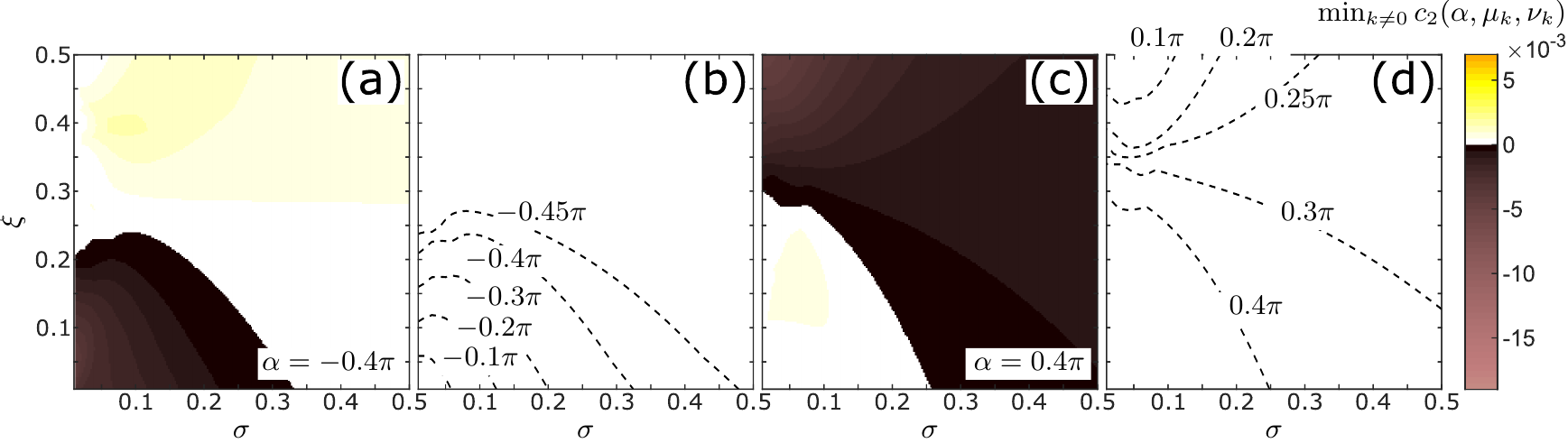}
	\caption{Stability analysis of the synchronous state of system~\eqref{eq:AKS_DDP_phi}--\eqref{eq:AKS_DDP_kappa} with plasticity rule~\eqref{eq:adaptationfunctionSymmetricOnRing}--\eqref{eq:STDP_DDP_betaShift} and coupling structure \eqref{eq:GaussianWeightNetwork}. Panels (a,c) show the minimum over all $k\ne 0< N$ of $c_2$, se ~\eqref{eq:synyStabCrit2}, for two different values of $\alpha$ depending on the mean value $\xi$ and the standard deviation $\sigma$ of the weight distribution. The minima are displayed in color code. Panels (b,d) show the boundaries between stable and unstable regions in $(\sigma,\xi)$-plane for different values of $\alpha$ as given in the figure. Parameters: (a) $N=400$, $\alpha=-0.4\pi$, (c) $N=400$, $\alpha=0.4\pi$, (b,c) $N=200$.}\label{fig:StabGaussianWeights}
\end{figure}
As we know from~\eqref{eq:synyStabCrit1}--\eqref{eq:synyStabCrit2}, for $\epsilon\ll 1$, the values of $c_2(\alpha,\mu_k,\nu_k)$ determine the stability of the synchronous state. In particular, the synchronous state is stable if $c_\text{min}=\min_{k\in{1,N-1}}c_2(\alpha,\mu_k,\nu_k)>0$ for a given $N$ and unstable otherwise. In figure~\ref{fig:StabGaussianWeights}(a), we display $c_\text{min}$ for $\alpha=-0.4\pi$ and different mean values $\xi$ and standard deviations $\sigma$ of the weight distribution. In agreement with the finding in section~\ref{sec:NonLocalRing}, the synchronized state stabilizes due to an increase of long distance interaction expressed by an increase of $\sigma$. Complementing the finding in section~\ref{sec:NonLocalRing}, here, we note that the stability can be also achieved by distributions with peaks at long distance links alone. In this case, the width of the distribution is not important. Figure~\ref{fig:StabGaussianWeights}(b) shows how the boundary between regions corresponding to stable and unstable synchronization change for different values of $\alpha$. As in the case of nonlocally coupled ring networks, with $\alpha\to 0$ (with $\alpha<0$) the boundary tends to the limiting point $(\sigma,\xi)=(0,0)$. On the contrary, if $\alpha\to-\pi/2$ (with $\alpha>-\pi/2$), the width of the distribution has to increase to have stable synchronization for small values of the mean $\xi$. 

An opposite scenario is shown in figure~\ref{fig:StabGaussianWeights}(c) for $\alpha=0.4\pi$. Here, an increase of the weights for long distance links destabilizes the synchronous state, as in figure~\ref{fig:SyncStability_NonLocalRing}(d,h,l). We also note that for small values $\alpha$, the synchronous state is stable for almost all values of $\sigma$ and $\xi$, see figure~\ref{fig:StabGaussianWeights}(d). Only in cases of distribution sharply peaked at long distances, i.e., $\xi$ close to $1/2$ and $\sigma$ close to $0$, the synchronous state is unstable. This effect could not be found in networks with nonlocally coupled rings, see section~\ref{sec:NonLocalRing}.
\section{Conclusions}\label{sec:conclusions}

In summary, we have investigated the phenomenon of synchronization on adaptive networks with heterogeneous plasticity rules. In particular, we have modeled systems with distance-dependent plasticity as they have been found in neuronal networks experimentally~\cite{FRO05a,LET06,SJO06,FRO10} as well as computational models~\cite{MEI20a}. For the realization, we have used a ring-like network architecture and associated the distance of two nodes with the distance of their placement on the ring.

In section~\ref{sec:msf}, we have developed a generalized master stability approach for phase oscillator models that are adaptively coupled and where each link has its own adaptation rule (plasticity). By using an explicit splitting of the time scales between fast dynamics of the phase oscillators and slow dynamics of the link weights, we have established an explicit stability condition for the synchronous state. More precisely, we found that the stability is governed by the coupling function and the eigenvalues of two structure matrices. These structure matrices $L^h$ and $L^{\mathrm{D}h}$ are determined by the connectivity of the network and the plasticity rules of the link weights. Note that for the structural matrices, the plasticity rule needs only to be known in the vicinity of $0$, which greatly facilitates the application of the approach to realistic forms of synaptic plasticity. Thus, we have extended previous work on the master stability function of adaptive networks~\cite{BER20b,VOC21} and broaden the scope of potential future applications for this methodology.

In section~\ref{sec:systDDP}, we applied the novel technique to a system of adaptively coupled oscillators with distance-dependent plasticity. Here, we have used a ring-like network structure to study the impact of long- and short-distance connections on the stability of synchronization. For this purpose we introduced an approximation of the eigenvalues for the structure matrices in section~\ref{sec:EVApprox}. This approximation allows for a comprehensive analysis of the stability as a function of various system parameters. Moreover, it enables us to identify critical eigenvalues that govern the stability of the synchronous state. In sections~\ref{sec:NonLocalRing} and~\ref{sec:Gaussian}, we have brought together all methodological findings and applied them to systems with a nonlocally coupled ring structure and with a Gaussian distribution of link weights. The latter structure accounts for the fact that in realistic neuronal populations the number of links with different distances are not uniformly distributed~\cite{LET06}. We found that long-distance connections can stabilize or destabilize the synchronous state, depending on the coupling function between the oscillators. A remarkable fact with respect to neuronal applications relates to the destabilization scenario. Here we observed that the destabilization can be attributed to the pronounced change of the plasticity rule from Hebbian to anti-Hebbian. For more realistic connectivity structures, we found that weight distributions of the connectivity structure with sharp peaks at long distances lead to destabilization for a wide range of the coupling function.

All in all, in this article, we have provided a general framework to study the emergence of synchronization in neuronal system with a heterogeneous plasticity rule. The developed methodology is not limited to distance-dependent types of plasticity and can also be used for non-symmetric setups. For the latter case, we have provided the necessary analytical result. In this work, we have restricted our attention to the case of phase oscillators, but the methods can be extended to more realistic neuron models by using techniques established, for example, in~\cite{BER20b}. Moreover, techniques are available that allow for further generalization towards systems with slightly different local dynamics at each node~\cite{SUN09a}. On the one hand, the master stability approach offers a great tool to study the stability of the synchronous state depending on the networks structure. On the other hand, this approach allows for characterizing the network structures that are, in some sense, optimal for synchronization~\cite{SKA14,ACH15}. In this regard, it remains an open question as to how plasticity optimizes the synchronizability of the network in a self-organized way. In addition, recent studies have shown that there is a great interest in synchronization phenomena to understand diseases such as Parkinson's disease~\cite{KRO20,KRO20a,KHA21} or epilepsy~\cite{OLM19,GER20} for the development of proper therapeutic treatments. We believe that our work provides an important step toward understanding synchronization under realistic conditions.

\section*{Conflict of Interest Statement}

The authors declare that the research was conducted in the absence of any commercial or financial relationships that could be construed as a potential conflict of interest.

\section*{Author Contributions}

RB designed the study and did the numerical simulations. RB and SY developed the analytical results. Both authors contributed to the preparation of the manuscript. Both authors read and approved the final manuscript.

\section*{Funding}
This work was supported by the German Research Foundation DFG, Project Nos. 411803875 and 440145547, and the Open Access Publication Fund of TU Berlin.


\begin{thebibliography}{94}
	\expandafter\ifx\csname natexlab\endcsname\relax\def\natexlab#1{#1}\fi
	\expandafter\ifx\csname urlstyle\endcsname\relax
	\expandafter\ifx\csname doi\endcsname\relax
	\def\doi#1{doi:\discretionary{}{}{}#1}\fi \else
	\expandafter\ifx\csname doi\endcsname\relax
	\def\doi{doi:\discretionary{}{}{}\begingroup \urlstyle{rm}\Url}\fi \fi
	\expandafter\ifx\csname selectlanguage\endcsname\relax
	\def\selectlanguage#1{}\fi
	
	\bibitem[{Newman(2003)}]{NEW03}
	Newman MEJ.
	\newblock The structure and function of complex networks.
	\newblock {\em SIAM Review\/} {\bf 45} (2003) 167--256.
	\newblock \doi{10.1137/s0036144503}.
	
	\bibitem[{Pikovsky et~al.(2001)Pikovsky, Rosenblum, and Kurths}]{PIK01}
	Pikovsky A, Rosenblum M, Kurths J.
	\newblock {\em Synchronization: a universal concept in nonlinear sciences\/}
	(Cambridge: Cambridge University Press), 1st edn. (2001).
	
	\bibitem[{Strogatz(2001)}]{STR01a}
	Strogatz SH.
	\newblock Exploring complex networks.
	\newblock {\em Nature\/} {\bf 410} (2001) 268--276.
	\newblock \doi{10.1038/35065725}.
	
	\bibitem[{Arenas et~al.(2008)Arenas, D{\'i}az-Guilera, Kurths, Moreno, and
		Zhou}]{ARE08}
	Arenas A, D{\'i}az-Guilera A, Kurths J, Moreno Y, Zhou C.
	\newblock Synchronization in complex networks.
	\newblock {\em Phys. Rep.\/} {\bf 469} (2008) 93--153.
	\newblock \doi{doi: 10.1016/j.physrep.2008.09.002}.
	
	\bibitem[{Boccaletti et~al.(2018)Boccaletti, Pisarchik, del Genio, and
		Amann}]{BOC18}
	Boccaletti S, Pisarchik AN, del Genio CI, Amann A.
	\newblock {\em Synchronization: {F}rom Coupled Systems to Complex Networks\/}
	(Cambridge: Cambridge University Press) (2018).
	
	\bibitem[{Kuramoto(1984)}]{KUR84}
	Kuramoto Y.
	\newblock {\em Chemical Oscillations, Waves and Turbulence\/} (Berlin:
	Springer-Verlag) (1984).
	
	\bibitem[{Pecora et~al.(1997)Pecora, Carroll, Johnson, Mar, and Heagy}]{PEC97}
	Pecora LM, Carroll TL, Johnson GA, Mar DJ, Heagy JF.
	\newblock Fundamentals of synchronization in chaotic systems, concepts, and
	applications.
	\newblock {\em Chaos\/} {\bf 7} (1997) 520--543.
	\newblock \doi{10.1063/1.166278}.
	
	\bibitem[{Yanchuk et~al.(2001)Yanchuk, Maistrenko, and Mosekilde}]{YAN01a}
	Yanchuk S, Maistrenko Y, Mosekilde E.
	\newblock {Partial synchronization and clustering in a system of diffusively
		coupled chaotic oscillators}.
	\newblock {\em Math. Comp. Simul.\/} {\bf 54} (2001) 491--508.
	
	\bibitem[{Choe et~al.(2010)Choe, Dahms, H{\"o}vel, and Sch{\"o}ll}]{CHO09}
	Choe CU, Dahms T, H{\"o}vel P, Sch{\"o}ll E.
	\newblock Controlling synchrony by delay coupling in networks: from in-phase to
	splay and cluster states.
	\newblock {\em Phys. Rev. E\/} {\bf 81} (2010) 025205(R).
	\newblock \doi{10.1103/physreve.81.025205}.
	
	\bibitem[{Belykh and Hasler(2011)}]{BEL11a}
	Belykh IV, Hasler M.
	\newblock Mesoscale and clusters of synchrony in networks of bursting neurons.
	\newblock {\em Chaos\/} {\bf 21} (2011) 016106.
	\newblock \doi{10.1063/1.3563581}.
	
	\bibitem[{Zhang and Motter(2020)}]{ZHA20}
	Zhang Y, Motter AE.
	\newblock Symmetry-independent stability analysis of synchronization patterns.
	\newblock {\em SIAM Rev.\/} {\bf 62} (2020) 817--836.
	\newblock \doi{10.1137/19m127358x}.
	
	\bibitem[{Berner et~al.(2019{\natexlab{a}})Berner, Sch{\"o}ll, and
		Yanchuk}]{BER19}
	Berner R, Sch{\"o}ll E, Yanchuk S.
	\newblock Multiclusters in networks of adaptively coupled phase oscillators.
	\newblock {\em SIAM J. Appl. Dyn. Syst.\/} {\bf 18} (2019{\natexlab{a}})
	2227--2266.
	\newblock \doi{10.1137/18m1210150}.
	
	\bibitem[{Jaros et~al.(2018)Jaros, Brezetsky, Levchenko, Dudkowski, Kapitaniak,
		and Maistrenko}]{JAR18}
	Jaros P, Brezetsky S, Levchenko R, Dudkowski D, Kapitaniak T, Maistrenko Y.
	\newblock Solitary states for coupled oscillators with inertia.
	\newblock {\em Chaos\/} {\bf 28} (2018) 011103.
	
	\bibitem[{Teichmann and Rosenblum(2019)}]{TEI19}
	Teichmann E, Rosenblum M.
	\newblock Solitary states and partial synchrony in oscillatory ensembles with
	attractive and repulsive interactions.
	\newblock {\em Chaos\/} {\bf 29} (2019) 093124.
	\newblock \doi{10.1063/1.5118843}.
	
	\bibitem[{Berner et~al.(2020{\natexlab{a}})Berner, Polanska, Sch{\"o}ll, and
		Yanchuk}]{BER20c}
	Berner R, Polanska A, Sch{\"o}ll E, Yanchuk S.
	\newblock Solitary states in adaptive nonlocal oscillator networks.
	\newblock {\em Eur. Phys. J. Spec. Top.\/} {\bf 229} (2020{\natexlab{a}})
	2183--2203.
	\newblock \doi{https://doi.org/10.1140/epjst/e2020-900253-0}.
	
	\bibitem[{Kuramoto and Battogtokh(2002)}]{KUR02a}
	Kuramoto Y, Battogtokh D.
	\newblock {Coexistence of Coherence and Incoherence in Nonlocally Coupled Phase
		Oscillators.}
	\newblock {\em Nonlin. Phen. in Complex Sys.\/} {\bf 5} (2002) 380--385.
	
	\bibitem[{Abrams and Strogatz(2004)}]{ABR04}
	Abrams DM, Strogatz SH.
	\newblock Chimera states for coupled oscillators.
	\newblock {\em Phys. Rev. Lett.\/} {\bf 93} (2004) 174102.
	\newblock \doi{10.1103/physrevlett.93.174102}.
	
	\bibitem[{Sch{\"o}ll(2016)}]{SCH16b}
	Sch{\"o}ll E.
	\newblock Synchronization patterns and chimera states in complex networks:
	interplay of topology and dynamics.
	\newblock {\em Eur. Phys. J. Spec. Top.\/} {\bf 225} (2016) 891--919.
	\newblock \doi{10.1140/epjst/e2016-02646-3}.
	
	\bibitem[{Omel'chenko(2018)}]{OME18a}
	Omel'chenko OE.
	\newblock The mathematics behind chimera states.
	\newblock {\em Nonlinearity\/} {\bf 31} (2018) R121.
	\newblock \doi{10.1088/1261-6544/aaaa07}.
	
	\bibitem[{Omel'chenko and Knobloch(2019)}]{OME19c}
	Omel'chenko OE, Knobloch E.
	\newblock Chimerapedia: coherence--incoherence patterns in one, two and three
	dimensions.
	\newblock {\em New J. Phys.\/} {\bf 21} (2019) 093034.
	\newblock \doi{10.1088/1367-2630/ab3f6b}.
	
	\bibitem[{Singer(1999)}]{SIN99}
	Singer W.
	\newblock {Neuronal Synchrony: A Versatile Code Review for the Definition of
		Relations?}
	\newblock {\em Neuron\/} {\bf 24} (1999) 49--65.
	
	\bibitem[{Fell and Axmacher(2011)}]{FEL11}
	Fell J, Axmacher N.
	\newblock The role of phase synchronization in memory processes.
	\newblock {\em Nat. Rev. Neurosci.\/} {\bf 12} (2011) 105--118.
	\newblock \doi{10.1038/nrn2979}.
	
	\bibitem[{Hammond et~al.(2007)Hammond, Bergman, and Brown}]{HAM07}
	Hammond C, Bergman H, Brown P.
	\newblock Pathological synchronization in {P}arkinson's disease: networks,
	models and treatments.
	\newblock {\em Trends Neurosci.\/} {\bf 30} (2007) 357--364.
	
	\bibitem[{Goriely et~al.(2020)Goriely, Kuhl, and Bick}]{GOR20}
	Goriely A, Kuhl E, Bick C.
	\newblock Neuronal oscillations on evolving networks: Dynamics, damage,
	degradation, decline, dementia, and death.
	\newblock {\em Phys. Rev. Lett.\/} {\bf 125} (2020) 128102.
	\newblock \doi{10.1103/physrevlett.125.128102}.
	
	\bibitem[{Pfeifer et~al.(2021)Pfeifer, Kromer, Cook, Hornbeck, Lim, Mortimer
		et~al.}]{PFE21}
	Pfeifer KJ, Kromer JA, Cook AJ, Hornbeck T, Lim EA, Mortimer BJP, et~al.
	\newblock {Coordinated Reset Vibrotactile Stimulation Induces Sustained
		Cumulative Benefits in Parkinson's Disease}.
	\newblock {\em Front. Physiol.\/} {\bf 12} (2021) 624317.
	\newblock \doi{10.3389/fphys.2021.624317}.
	
	\bibitem[{Jiruska et~al.(2013)Jiruska, de~Curtis, Jefferys, Schevon, Schiff,
		and Schindler}]{JIR13}
	Jiruska P, de~Curtis M, Jefferys JGR, Schevon CA, Schiff SJ, Schindler K.
	\newblock Synchronization and desynchronization in epilepsy: controversies and
	hypotheses.
	\newblock {\em J. Physiol.\/} {\bf 591.4} (2013) 787--797.
	
	\bibitem[{Jirsa et~al.(2014)Jirsa, Stacey, Quilichini, Ivanov, and
		Bernard}]{JIR14}
	Jirsa VK, Stacey WC, Quilichini PP, Ivanov AI, Bernard C.
	\newblock On the nature of seizure dynamics.
	\newblock {\em Brain\/} {\bf 137} (2014) 2210.
	
	\bibitem[{Andrzejak et~al.(2016)Andrzejak, Rummel, Mormann, and
		Schindler}]{AND16}
	Andrzejak RG, Rummel C, Mormann F, Schindler K.
	\newblock All together now: Analogies between chimera state collapses and
	epileptic seizures.
	\newblock {\em Sci. Rep.\/} {\bf 6} (2016) 23000.
	\newblock \doi{10.1038/srep23000}.
	
	\bibitem[{Gerster et~al.(2020)Gerster, Berner, Sawicki, Zakharova, Skoch,
		Hlinka et~al.}]{GER20}
	Gerster M, Berner R, Sawicki J, Zakharova A, Skoch A, Hlinka J, et~al.
	\newblock {FitzHugh-Nagumo} oscillators on complex networks mimic
	epileptic-seizure-related synchronization phenomena.
	\newblock {\em Chaos\/} {\bf 30} (2020) 123130.
	\newblock \doi{10.1063/5.0021420}.
	
	\bibitem[{Tass et~al.(2012)Tass, Adamchic, Freund, von Stackelberg, and
		Hauptmann}]{TAS12}
	Tass PA, Adamchic I, Freund HJ, von Stackelberg T, Hauptmann C.
	\newblock Counteracting tinnitus by acoustic coordinated reset neuromodulation.
	\newblock {\em Restor. Neurol. Neurosci.\/} {\bf 30} (2012) 137--159.
	
	\bibitem[{Tass and Popovych(2012)}]{TAS12a}
	Tass PA, Popovych OV.
	\newblock Unlearning tinnitus-related cerebral synchrony with acoustic
	coordinated reset stimulation: theoretical concept and modelling.
	\newblock {\em Biol. Cybern.\/} {\bf 106} (2012) 27--36.
	\newblock \doi{10.1007/s00422-012-0479-5}.
	
	\bibitem[{Uhlhaas et~al.(2009)Uhlhaas, Pipa, Lima, Melloni, Neuenschwander,
		Nikolic et~al.}]{UHL09}
	Uhlhaas P, Pipa G, Lima B, Melloni L, Neuenschwander S, Nikolic D, et~al.
	\newblock Neural synchrony in cortical networks: history, concept and current
	status.
	\newblock {\em Front. Integr. Neurosci.\/} {\bf 3} (2009) 17.
	\newblock \doi{10.3389/neuro.07.017.2009}.
	
	\bibitem[{Pecora and Carroll(1998)}]{PEC98}
	Pecora LM, Carroll TL.
	\newblock {M}aster {S}tability {F}unctions for {S}ynchronized {C}oupled
	{S}ystems.
	\newblock {\em Phys. Rev. Lett.\/} {\bf 80} (1998) 2109--2112.
	\newblock \doi{10.1103/physrevlett.80.2109}.
	
	\bibitem[{Flunkert et~al.(2010)Flunkert, Yanchuk, Dahms, and
		Sch{\"o}ll}]{FLU10b}
	Flunkert V, Yanchuk S, Dahms T, Sch{\"o}ll E.
	\newblock Synchronizing distant nodes: a universal classification of networks.
	\newblock {\em Phys. Rev. Lett.\/} {\bf 105} (2010) 254101.
	\newblock \doi{10.1103/physrevlett.105.254101}.
	
	\bibitem[{Dahms et~al.(2012)Dahms, Lehnert, and Sch{\"o}ll}]{DAH12}
	Dahms T, Lehnert J, Sch{\"o}ll E.
	\newblock Cluster and group synchronization in delay-coupled networks.
	\newblock {\em Phys. Rev. E\/} {\bf 86} (2012) 016202.
	\newblock \doi{10.1103/physreve.86.016202}.
	
	\bibitem[{Keane et~al.(2012)Keane, Dahms, Lehnert, Suryanarayana, H{\"o}vel,
		and Sch{\"o}ll}]{KEA12}
	Keane A, Dahms T, Lehnert J, Suryanarayana SA, H{\"o}vel P, Sch{\"o}ll E.
	\newblock Synchronisation in networks of delay-coupled {type-I} excitable
	systems.
	\newblock {\em Eur. Phys. J. B\/} {\bf 85} (2012) 407.
	\newblock \doi{10.1140/epjb/e2012-30810-x}.
	
	\bibitem[{Kyrychko et~al.(2014)Kyrychko, Blyuss, and Sch{\"o}ll}]{KYR14}
	Kyrychko YN, Blyuss KB, Sch{\"o}ll E.
	\newblock Synchronization of networks of oscillators with distributed-delay
	coupling.
	\newblock {\em Chaos\/} {\bf 24} (2014) 043117.
	\newblock \doi{10.1063/1.4898771}.
	
	\bibitem[{Lehnert(2016)}]{LEH15b}
	Lehnert J.
	\newblock {\em Controlling synchronization patterns in complex networks\/}.
	\newblock Springer Theses (Heidelberg: Springer) (2016).
	
	\bibitem[{Tang et~al.(2019)Tang, Wu, L{\"u}, Lu, and D'Souza}]{TAN19}
	Tang L, Wu X, L{\"u} J, Lu J, D'Souza RM.
	\newblock Master stability functions for complete, intralayer, and interlayer
	synchronization in multiplex networks of coupled r{\"o}ssler oscillators.
	\newblock {\em Phys. Rev. E\/} {\bf 99} (2019).
	\newblock 012304.
	
	\bibitem[{Berner et~al.(2020{\natexlab{b}})Berner, Sawicki, and
		Sch{\"o}ll}]{BER20}
	Berner R, Sawicki J, Sch{\"o}ll E.
	\newblock Birth and stabilization of phase clusters by multiplexing of adaptive
	networks.
	\newblock {\em Phys. Rev. Lett.\/} {\bf 124} (2020{\natexlab{b}}) 088301.
	\newblock \doi{10.1103/physrevlett.124.088301}.
	
	\bibitem[{B{\"o}rner et~al.(2020)B{\"o}rner, Schultz, {\"U}nzelmann, Wang,
		Hellmann, and Kurths}]{BOE20}
	B{\"o}rner R, Schultz P, {\"U}nzelmann B, Wang D, Hellmann F, Kurths J.
	\newblock Delay master stability of inertial oscillator networks.
	\newblock {\em Phys. Rev. Research\/} {\bf 2} (2020) 023409.
	\newblock \doi{10.1103/physrevresearch.2.023409}.
	
	\bibitem[{Mulas et~al.(2020)Mulas, Kuehn, and Jost}]{MUL20}
	Mulas R, Kuehn C, Jost J.
	\newblock Coupled dynamics on hypergraphs: Master stability of steady states
	and synchronization.
	\newblock {\em Phys. Rev. E\/} {\bf 101} (2020) 062313.
	\newblock \doi{10.1103/physreve.101.062313}.
	
	\bibitem[{Belykh et~al.(2004)Belykh, Belykh, and Hasler}]{BEL04}
	Belykh VN, Belykh IV, Hasler M.
	\newblock Connection graph stability method for synchronized coupled chaotic
	systems.
	\newblock {\em Physica D\/} {\bf 195} (2004) 159--187.
	
	\bibitem[{Belykh et~al.(2005)Belykh, de~Lange, and Hasler}]{BEL05b}
	Belykh IV, de~Lange E, Hasler M.
	\newblock Synchronization of bursting neurons: What matters in the network
	topology.
	\newblock {\em Phys. Rev. Lett.\/} {\bf 94} (2005) 188101.
	\newblock \doi{10.1103/physrevlett.94.188101}.
	
	\bibitem[{Belykh et~al.(2006{\natexlab{a}})Belykh, Belykh, and Hasler}]{BEL06}
	Belykh IV, Belykh VN, Hasler M.
	\newblock Generalized connection graph method for synchronization in
	asymmetrical networks.
	\newblock {\em Physica D\/} {\bf 224} (2006{\natexlab{a}}) 42--51.
	\newblock \doi{doi: 10.1016/j.physd.2006.09.014}.
	
	\bibitem[{Belykh et~al.(2006{\natexlab{b}})Belykh, Belykh, and Hasler}]{BEL06a}
	Belykh IV, Belykh VN, Hasler M.
	\newblock Synchronization in asymmetrically coupled networks with node balance.
	\newblock {\em Chaos\/} {\bf 16} (2006{\natexlab{b}}) 015102.
	
	\bibitem[{Daley et~al.(2020)Daley, Zhao, and Belykh}]{DAL20}
	Daley K, Zhao K, Belykh IV.
	\newblock Synchronizability of directed networks: {T}he power of non-existent
	ties.
	\newblock {\em Chaos\/} {\bf 30} (2020) 043102.
	
	\bibitem[{Berner et~al.(2021{\natexlab{a}})Berner, Vock, Sch{\"o}ll, and
		Yanchuk}]{BER20b}
	Berner R, Vock S, Sch{\"o}ll E, Yanchuk S.
	\newblock Desynchronization transitions in adaptive networks.
	\newblock {\em Phys. Rev. Lett.\/} {\bf 126} (2021{\natexlab{a}}) 028301.
	\newblock \doi{10.1103/physrevlett.126.028301}.
	
	\bibitem[{Jain and Krishna(2001)}]{JAI01}
	Jain S, Krishna S.
	\newblock A model for the emergence of cooperation, interdependence, and
	structure in evolving networks.
	\newblock {\em Proc. Natl. Acad. Sci.\/} {\bf 98} (2001) 543--547.
	\newblock \doi{10.1073/pnas.98.2.543}.
	
	\bibitem[{Proulx et~al.(2005)Proulx, Promislow, and Phillips}]{PRO05a}
	Proulx SR, Promislow DEL, Phillips PC.
	\newblock Network thinking in ecology and evolution.
	\newblock {\em Trends Ecol. Evol.\/} {\bf 20} (2005) 345--353.
	\newblock \doi{10.1016/j.tree.2005.04.004}.
	
	\bibitem[{Gross et~al.(2006)Gross, D'Lima, and Blasius}]{GRO06b}
	Gross T, D'Lima CJD, Blasius B.
	\newblock Epidemic dynamics on an adaptive network.
	\newblock {\em Phys. Rev. Lett.\/} {\bf 96} (2006) 208701.
	\newblock \doi{10.1103/physrevlett.96.208701}.
	
	\bibitem[{Martens and Klemm(2017)}]{MAR17b}
	Martens EA, Klemm K.
	\newblock Transitions from trees to cycles in adaptive flow networks.
	\newblock {\em Front. Phys.\/} {\bf 5} (2017) 62.
	\newblock \doi{10.3389/fphy.2017.00062}.
	
	\bibitem[{Kuehn(2019)}]{KUE19a}
	Kuehn C.
	\newblock Multiscale dynamics of an adaptive catalytic network.
	\newblock {\em Math. Model. Nat. Phenom.\/} {\bf 14} (2019) 402.
	\newblock \doi{10.1051/mmnp/2019015}.
	
	\bibitem[{Horstmeyer and Kuehn(2020)}]{HOR20}
	Horstmeyer L, Kuehn C.
	\newblock Adaptive voter model on simplicial complexes.
	\newblock {\em Phys. Rev. E\/} {\bf 101} (2020) 022305.
	\newblock \doi{10.1103/physreve.101.022305}.
	
	\bibitem[{Meisel and Gross(2009)}]{MEI09a}
	Meisel C, Gross T.
	\newblock Adaptive self-organization in a realistic neural network model.
	\newblock {\em Phys. Rev. E\/} {\bf 80} (2009) 061917.
	\newblock \doi{10.1103/physreve.80.061917}.
	
	\bibitem[{Mikkelsen et~al.(2013)Mikkelsen, Imparato, and Torcini}]{MIK13}
	Mikkelsen K, Imparato A, Torcini A.
	\newblock Emergence of slow collective oscillations in neural networks with
	spike-timing dependent plasticity.
	\newblock {\em Phys. Rev. Lett.\/} {\bf 110} (2013) 208101.
	
	\bibitem[{Mikkelsen et~al.(2014)Mikkelsen, Imparato, and Torcini}]{MIK14}
	Mikkelsen K, Imparato A, Torcini A.
	\newblock Sisyphus effect in pulse-coupled excitatory neural networks with
	spike-timing-dependent plasticity.
	\newblock {\em Phys. Rev. E\/} {\bf 89} (2014) 062701.
	\newblock \doi{10.1103/physreve.89.062701}.
	
	\bibitem[{Markram et~al.(1997)Markram, L\"ubke, Frotscher, and
		Sakmann}]{MAR97a}
	Markram H, L\"ubke J, Frotscher M, Sakmann B.
	\newblock Regulation of synaptic efficacy by coincidence of postsynaptic {AP}s
	and {EPSP}s.
	\newblock {\em Science\/} {\bf 275} (1997) 213--215.
	\newblock \doi{10.1126/science.275.5297.213}.
	
	\bibitem[{Abbott and Nelson(2000)}]{ABB00}
	Abbott LF, Nelson S.
	\newblock Synaptic plasticity: taming the beast.
	\newblock {\em Nat. Neurosci.\/} {\bf 3} (2000) 1178--1183.
	\newblock \doi{10.1038/81453}.
	
	\bibitem[{Caporale and Dan(2008)}]{CAP08a}
	Caporale N, Dan Y.
	\newblock Spike timing-dependent plasticity: A {H}ebbian learning rule.
	\newblock {\em Annu. Rev. Neurosci.\/} {\bf 31} (2008) 25--46.
	\newblock \doi{10.1146/annurev.neuro.31.060407.125639}.
	
	\bibitem[{Popovych et~al.(2013)Popovych, Yanchuk, and Tass}]{POP13}
	Popovych OV, Yanchuk S, Tass PA.
	\newblock Self-organized noise resistance of oscillatory neural networks with
	spike timing-dependent plasticity.
	\newblock {\em Sci. Rep.\/} {\bf 3} (2013) 2926.
	\newblock \doi{10.1038/srep02926}.
	
	\bibitem[{Zenke et~al.(2015)Zenke, Agnes, and Gerstner}]{ZEN15}
	Zenke F, Agnes EJ, Gerstner W.
	\newblock Diverse synaptic plasticity mechanisms orchestrated to form and
	retrieve memories in spiking neural networks.
	\newblock {\em Nat. Commun.\/} {\bf 6} (2015).
	\newblock \doi{10.1038/ncomms7922}.
	
	\bibitem[{Tazerart et~al.(2020)Tazerart, Mitchell, Miranda-Rottmann, and
		Araya}]{TAZ20}
	Tazerart S, Mitchell DE, Miranda-Rottmann S, Araya R.
	\newblock A spike-timing-dependent plasticity rule for dendritic spines.
	\newblock {\em Nat. Commun.\/} {\bf 11} (2020) 4276.
	\newblock \doi{10.1038/s41467-020-17861-7}.
	
	\bibitem[{Froemke et~al.(2005)Froemke, Poo, and Dan}]{FRO05a}
	Froemke RC, Poo Mm, Dan Y.
	\newblock Spike-timing-dependent synaptic plasticity depends on dendritic
	location.
	\newblock {\em Nature\/} {\bf 434} (2005) 221.
	\newblock \doi{10.1038/nature03366}.
	
	\bibitem[{Sj{\"o}str{\"o}m and H\"{a}usser(2006)}]{SJO06}
	Sj{\"o}str{\"o}m PJ, H\"{a}usser M.
	\newblock A {C}ooperative {S}witch {D}etermines the {S}ign of {S}ynaptic
	{P}lasticity in {D}istal {D}endrites of {N}eocortical {P}yramidal {N}eurons.
	\newblock {\em Neuron\/} {\bf 51} (2006) 227.
	\newblock \doi{10.1016/j.neuron.2006.06.017}.
	
	\bibitem[{Froemke et~al.(2010)Froemke, Letzkus, Kampa, Hang, and
		Stuart}]{FRO10}
	Froemke RC, Letzkus JJ, Kampa BM, Hang GB, Stuart GJ.
	\newblock Dendritic synapse location and neocortical spike-timing-dependent
	plasticity.
	\newblock {\em Front. Synaptic Neurosci.\/} {\bf 2} (2010).
	\newblock \doi{10.3389/fnsyn.2010.00029}.
	
	\bibitem[{Letzkus et~al.(2006)Letzkus, Kampa, and Stuart}]{LET06}
	Letzkus JJ, Kampa BM, Stuart GJ.
	\newblock Learning {R}ules for {S}pike {T}iming-{D}ependent {P}lasticity
	{D}epend on {D}endritic {S}ynapse {L}ocation.
	\newblock {\em J. Neurosci.\/} {\bf 26} (2006) 10420.
	\newblock \doi{10.1523/jneurosci.2650-06.2006}.
	
	\bibitem[{Meissner-Bernard et~al.(2020)Meissner-Bernard, Tsai, Logiaco, and
		Gerstner}]{MEI20a}
	Meissner-Bernard C, Tsai MC, Logiaco L, Gerstner W.
	\newblock Dendritic {V}oltage {R}ecordings {E}xplain {P}aradoxical {S}ynaptic
	{P}lasticity: {A} {M}odeling {S}tudy.
	\newblock {\em Front. Synaptic Neurosci.\/} {\bf 12} (2020) 585539.
	\newblock \doi{10.3389/fnsyn.2020.585539}.
	
	\bibitem[{Aoki and Aoyagi(2009)}]{AOK09}
	Aoki T, Aoyagi T.
	\newblock Co-evolution of phases and connection strengths in a network of phase
	oscillators.
	\newblock {\em Phys. Rev. Lett.\/} {\bf 102} (2009) 034101.
	\newblock \doi{10.1103/physrevlett.102.034101}.
	
	\bibitem[{Kasatkin et~al.(2017)Kasatkin, Yanchuk, Sch{\"o}ll, and
		Nekorkin}]{KAS17}
	Kasatkin DV, Yanchuk S, Sch{\"o}ll E, Nekorkin VI.
	\newblock {S}elf-organized emergence of multi-layer structure and chimera
	states in dynamical networks with adaptive couplings.
	\newblock {\em Phys. Rev. E\/} {\bf 96} (2017) 062211.
	\newblock \doi{10.1103/physreve.96.062211}.
	
	\bibitem[{Kasatkin and Nekorkin(2018)}]{KAS18a}
	Kasatkin DV, Nekorkin VI.
	\newblock The effect of topology on organization of synchronous behavior in
	dynamical networks with adaptive couplings.
	\newblock {\em Eur. Phys. J. Spec. Top.\/} {\bf 227} (2018) 1051.
	
	\bibitem[{Berner et~al.(2019{\natexlab{b}})Berner, Fialkowski, Kasatkin,
		Nekorkin, Yanchuk, and Sch{\"o}ll}]{BER19a}
	Berner R, Fialkowski J, Kasatkin DV, Nekorkin VI, Yanchuk S, Sch{\"o}ll E.
	\newblock Hierarchical frequency clusters in adaptive networks of phase
	oscillators.
	\newblock {\em Chaos\/} {\bf 29} (2019{\natexlab{b}}) 103134.
	\newblock \doi{10.1063/1.5097835}.
	
	\bibitem[{Berner et~al.(2021{\natexlab{b}})Berner, Yanchuk, and
		Sch{\"o}ll}]{BER21a}
	Berner R, Yanchuk S, Sch{\"o}ll E.
	\newblock What adaptive neuronal networks teach us about power grids.
	\newblock {\em Phys. Rev. E\/} {\bf 103} (2021{\natexlab{b}}) 042315.
	\newblock \doi{10.1103/physreve.103.042315}.
	
	\bibitem[{Feketa et~al.(2019)Feketa, Schaum, and Meurer}]{FEK20}
	Feketa P, Schaum A, Meurer T.
	\newblock Synchronization and multi-cluster capabilities of oscillatory
	networks with adaptive coupling.
	\newblock {\em IEEE Trans. Autom. Control\/}  (2019).
	\newblock \doi{10.1109/tac.2020.3012528}.
	
	\bibitem[{Franovi{\'c} et~al.(2020)Franovi{\'c}, Yanchuk, Eydam, Bacic, and
		Wolfrum}]{FRA20}
	Franovi{\'c} I, Yanchuk S, Eydam S, Bacic I, Wolfrum M.
	\newblock Dynamics of a stochastic excitable system with slowly adapting
	feedback.
	\newblock {\em Chaos\/} {\bf 30} (2020) 083109.
	\newblock \doi{10.1063/1.5145176}.
	\newblock ArXiv:2001.07650.
	
	\bibitem[{Popovych et~al.(2015)Popovych, Xenakis, and Tass}]{POP15}
	Popovych OV, Xenakis MN, Tass PA.
	\newblock The spacing principle for unlearning abnormal neuronal synchrony.
	\newblock {\em PLoS ONE\/} {\bf 10} (2015) e0117205.
	\newblock \doi{10.1371/journal.pone.0117205}.
	
	\bibitem[{L\"ucken et~al.(2016)L\"ucken, Popovych, Tass, and Yanchuk}]{LUE16}
	L\"ucken L, Popovych OV, Tass PA, Yanchuk S.
	\newblock {N}oise-enhanced coupling between two oscillators with long-term
	plasticity.
	\newblock {\em Phys. Rev. E\/} {\bf 93} (2016) 032210.
	\newblock \doi{10.1103/physreve.93.032210}.
	
	\bibitem[{Chakravartula et~al.(2017)Chakravartula, Indic, Sundaram, and
		Killingback}]{CHA17a}
	Chakravartula S, Indic P, Sundaram B, Killingback T.
	\newblock Emergence of local synchronization in neuronal networks with adaptive
	couplings.
	\newblock {\em PLoS ONE\/} {\bf 12} (2017) e0178975.
	\newblock \doi{10.1371/journal.pone.0178975}.
	
	\bibitem[{R{\"o}hr et~al.(2019)R{\"o}hr, Berner, Lameu, Popovych, and
		Yanchuk}]{ROE19a}
	R{\"o}hr V, Berner R, Lameu EL, Popovych OV, Yanchuk S.
	\newblock Frequency cluster formation and slow oscillations in neural
	populations with plasticity.
	\newblock {\em PLoS ONE\/} {\bf 14} (2019) e0225094.
	\newblock \doi{10.1371/journal.pone.0225094}.
	
	\bibitem[{Sakaguchi and Kuramoto(1986)}]{SAK86}
	Sakaguchi H, Kuramoto Y.
	\newblock A soluble active rotater model showing phase transitions via mutual
	entertainment.
	\newblock {\em Prog. Theor. Phys\/} {\bf 76} (1986) 576--581.
	
	\bibitem[{Madadi~Asl et~al.(2017)Madadi~Asl, Valizadeh, and Tass}]{ASL17}
	Madadi~Asl M, Valizadeh A, Tass PA.
	\newblock Dendritic and axonal propagation delays determine emergent structures
	of neuronal networks with plastic synapses.
	\newblock {\em Sci. Rep.\/} {\bf 7} (2017) 39682.
	\newblock \doi{10.1038/srep39682}.
	
	\bibitem[{Madadi~Asl et~al.(2018)Madadi~Asl, Valizadeh, and Tass}]{ASL18a}
	Madadi~Asl M, Valizadeh A, Tass PA.
	\newblock Dendritic and axonal propagation delays may shape neuronal networks
	with plastic synapses.
	\newblock {\em Front. Physiol.\/} {\bf 9} (2018) 1849.
	\newblock \doi{10.3389/fphys.2018.01849}.
	
	\bibitem[{Berner et~al.(2021{\natexlab{c}})Berner, Vock, Sch{\"o}ll, and
		Yanchuk}]{BER21b}
	Berner R, Vock S, Sch{\"o}ll E, Yanchuk S.
	\newblock Desynchronization transitions in adaptive networks.
	\newblock {\em Phys. Rev. Lett.\/} {\bf 126} (2021{\natexlab{c}}) 028301.
	\newblock \doi{10.1103/physrevlett.126.028301}.
	
	\bibitem[{Vock et~al.(2021)Vock, Berner, Yanchuk, and Sch{\"o}ll}]{VOC21}
	Vock S, Berner R, Yanchuk S, Sch{\"o}ll E.
	\newblock Effect of diluted connectivities on cluster synchronization of
	adaptively coupled oscillator networks.
	\newblock {\em Scientia Iranica D\/} {\bf 28} (2021).
	\newblock ArXiv:2101.05601.
	
	\bibitem[{Liesen and Mehrmann(2015)}]{LIE15}
	Liesen J, Mehrmann V.
	\newblock {\em Linear Algebra\/} (Cham: Springer) (2015).
	\newblock \doi{10.1007/978-3-319-24346-7}.
	
	\bibitem[{Gray(2006)}]{GRA06}
	Gray RM.
	\newblock Toeplitz and circulant matrices: A review.
	\newblock {\em Found. Trends Commun. Inf. Theory\/} (Hanover, MA, USA: Now
	Publishers Inc.), vol.~2 (2006), 155--239.
	
	\bibitem[{Aoki and Aoyagi(2011)}]{AOK11}
	Aoki T, Aoyagi T.
	\newblock Self-organized network of phase oscillators coupled by
	activity-dependent interactions.
	\newblock {\em Phys. Rev. E\/} {\bf 84} (2011) 066109.
	\newblock \doi{10.1103/physreve.84.066109}.
	
	\bibitem[{Sun et~al.(2009)Sun, Bollt, and Nishikawa}]{SUN09a}
	Sun J, Bollt EM, Nishikawa T.
	\newblock Master stability functions for coupled nearly identical dynamical
	systems.
	\newblock {\em Europhys. Lett.\/} {\bf 85} (2009) 60011.
	
	\bibitem[{Skardal et~al.(2014)Skardal, Taylor, and Sun}]{SKA14}
	Skardal PS, Taylor D, Sun J.
	\newblock Optimal synchronization in complex networks.
	\newblock {\em Phys. Rev. Lett.\/} {\bf 113} (2014) 144101.
	\newblock \doi{10.1103/physrevlett.113.144101}.
	
	\bibitem[{Acharyya and Amritkar(2015)}]{ACH15}
	Acharyya S, Amritkar RE.
	\newblock Synchronization of nearly identical dynamical systems: Size
	instability.
	\newblock {\em Phys. Rev. E\/} {\bf 92} (2015) 052902.
	\newblock \doi{10.1103/physreve.92.052902}.
	
	\bibitem[{Kromer and Tass(2020)}]{KRO20}
	Kromer JA, Tass PA.
	\newblock Long-lasting desynchronization by decoupling stimulation.
	\newblock {\em Phys. Rev. Research\/} {\bf 2} (2020) 033101.
	\newblock \doi{10.1103/physrevresearch.2.033101}.
	
	\bibitem[{Kromer et~al.(2020)Kromer, Khaledi-Nasab, and Tass}]{KRO20a}
	Kromer JA, Khaledi-Nasab A, Tass PA.
	\newblock Impact of number of stimulation sites on long-lasting
	desynchronization effects of coordinated reset stimulation.
	\newblock {\em Chaos\/} {\bf 30} (2020) 083134.
	\newblock \doi{10.1063/5.0015196}.
	
	\bibitem[{Khaledi-Nasab et~al.(2021)Khaledi-Nasab, Kromer, and Tass}]{KHA21}
	Khaledi-Nasab A, Kromer JA, Tass PA.
	\newblock Long-{L}asting {D}esynchronization of {P}lastic {N}eural {N}etworks
	by {R}andom {R}eset {S}timulation.
	\newblock {\em Front. Physiol.\/} {\bf 11} (2021) 622620.
	\newblock \doi{10.3389/fphys.2020.622620}.
	
	\bibitem[{Olmi et~al.(2019)Olmi, Petkoski, Guye, Bartolomei, and Jirsa}]{OLM19}
	Olmi S, Petkoski S, Guye M, Bartolomei F, Jirsa VK.
	\newblock Controlling seizure propagation in large-scale brain networks.
	\newblock {\em PLoS Comp. Biol.\/} {\bf 15} (2019) e1006805.
	
\end{thebibliography}
\end{document}